\documentclass[11pt, a4paper]{amsart}

\usepackage{amsfonts,amsmath,amssymb, amscd}

\usepackage[all]{xy}

\newtheorem{theorem}{Theorem}[section]
\newtheorem{lemma}[theorem]{Lemma}
\newtheorem{prop}[theorem]{Proposition}
\newtheorem{corollary}[theorem]{Corollary}

\theoremstyle{definition}
\newtheorem{definition}[theorem]{Definition}
\newtheorem{rem}[theorem]{Remark}

\newtheorem{example}[theorem]{Example}

\usepackage{mathrsfs}
\usepackage{color}

\newcommand\pf{\begin{proof}}
\newcommand\epf{\end{proof}}
\newcommand\tu{\textup}

\numberwithin{equation}{section}

\title[Quotients in super-symmetry]
{Quotients in super-symmetry: \\ formal supergroup case}

\author[Y.~Takahashi]{Yuta Takahashi}
\address{Yuta Takahashi,
Graduate School of Pure and Applied Sciences, 
University of Tsukuba, Ibaraki 305-8571, Japan}
\email{y-takahashi@math.tsukuba.ac.jp}

\author[A.~Masuoka]{Akira Masuoka}
\address{Akira Masuoka,
Department of Mathematics, 
University of Tsukuba, 
Ibaraki 305-8571, Japan}
\email{akira@math.tsukuba.ac.jp}

\begin{document}

\begin{abstract}
We describe the structure of the quotient $\mathfrak{G}/\mathfrak{H}$ of a formal supergroup $\mathfrak{G}$ by its formal sub-supergroup $\mathfrak{H}$. 
This is a consequence which arises as a continuation of the authors' work 
(partly with M. Hashi) on algebraic/analytic supergoups.
The results are presented and proved in terms 
of super-cocommutative Hopf superalgebras. The notion of 
co-free super-coalgebras plays a role, in particular. 
\end{abstract}

\maketitle



\section{Introduction}\label{Sec1}

\subsection{Background and objective}
The quotient $\mathfrak{G}/\mathfrak{H}$ of a group $\mathfrak{G}$ by its subgroup $\mathfrak{H}$ 
is the set of all $\mathfrak{H}$-orbits in $\mathfrak{G}$. This simple fact for abstract groups immediately turns
into a difficult question if we replace abstract groups with groups with geometric structure, such as algebraic
or analytic groups. Does the question become even more difficult when we consider those geometric groups 
in the generalized supersymmetric context, or namely, supergroups? 
The authors want to answer that it is not so difficult as was supposed, not so much more than that
for geometric (non-super) groups. For we can describe the structure
of the quotient $\mathfrak{G}/\mathfrak{H}$ for supergroups in terms of the non-super quotient 
$\mathfrak{G}_{\operatorname{ev}}/\mathfrak{H}_{\operatorname{ev}}$, where 
$\mathfrak{G}_{\operatorname{ev}}$ (resp., $\mathfrak{H}_{\operatorname{ev}}$) is the geometric (non-super)
group naturally associated with the supergroup $\mathfrak{G}$ (resp., $\mathfrak{H}$). 
Indeed, such a description has been given by the authors \cite[Theorem 4.12, Remark 4.13]{MT} for affine algebraic supergroups (see also \cite[Sect.\! 14.2]{MZ1} for non-affine case)
and by Hoshi and the authors \cite[Theorem 7.3, Remark 7.4 (1)]{HMT} for analytic supergroups. 

A main objective of this paper is to give such a description for formal supergroups. 
Those formal supergroups (or resp., formal superschemes) form a category which is equivalent to the category of super-cocommutative Hopf
superalgebras (resp., super-coalgebras), as is shown in Appendix below. 
In the text the results all are presented in terms of
such Hopf superalgebras and super-coalgebras, some of which will be translated into the language of formal super-objects in the Appendix.

\subsection{Main results}
Throughout in this paper we work over a fixed base field $\Bbbk$ whose characteristic 
$\operatorname{char} \Bbbk$ is not 2.
The unadorned $\otimes$ denotes the tensor product over $\Bbbk$. 
Except at the beginning part of the following Section 2 we always assume super-coalgebras and Hopf superalgebras to be super-cocommutative. 
Compare this situation with that in \cite{MT}, where the authors discussed super-commutative Hopf superalgebras
in order to investigate affine algebraic supergroups; some arguments of the present paper are in fact dual to
those of the cited article.

Let $\mathfrak{G}=\operatorname{Sp}^*(\mathbb{J})$ be a formal superscheme; the notation means
that it corresponds to a super-coalgebra $\mathbb{J}$. It is known that $\mathbb{J}$ includes the largest ordinary subcoalgebra, say, $J$. The formal scheme $\mathfrak{G}_{\mathrm{ev}}$ naturally associated to $\mathfrak{G}$
is the one corresponding to $J$, or in notation, $\mathfrak{G}_{\mathrm{ev}}=\operatorname{Sp}^*(J)$.
Suppose that $\mathfrak{G}$ is a formal supergroup, which means that $\mathbb{J}$ is a Hopf superalgebra.
Then $J$ is a Hopf subalgebra of $\mathbb{J}$. Let $\mathfrak{H}=\operatorname{Sp}^*(\mathbb{K})$
be a formal sub-supergroup of $\mathfrak{G}$, or equivalently, let $\mathbb{K}$ be a Hopf sub-superalgebra 
of $\mathbb{J}$. The associated formal group $\mathfrak{H}_{\mathrm{ev}}=\operatorname{Sp}^*(K)$
is then a formal subgroup of $\mathfrak{G}_{\mathrm{ev}}$, or equivalently, $K$ is a Hopf subalgebra of $J$.
The quotient formal superscheme $\mathfrak{G}/\mathfrak{H}$ corresponds to the quotient super-coalgebra
\[
\mathbb{J}/\hspace{-1mm}/\mathbb{K}:=\mathbb{J}/\mathbb{J}\mathbb{K}^+
\]
of $\mathbb{J}$ by the super-coideal $\mathbb{J}\mathbb{K}^+$ which is generated, as a left super-ideal, by
the augmentation super-ideal $\mathbb{K}^+$ of $\mathbb{K}$; or in notation, 
\[
\mathfrak{G}/\mathfrak{H}=\operatorname{Sp}^*(\mathbb{J}/\hspace{-1mm}/\mathbb{K}). 
\]
The associated formal scheme $(\mathfrak{G}/\mathfrak{H})_{\mathrm{ev}}$ is proved to be naturally isomorphic to
$\mathfrak{G}_{\mathrm{ev}}/\mathfrak{H}_{\mathrm{ev}}$ (see Corollary \ref{C41}), whence we have
\[
(\mathfrak{G}/\mathfrak{H})_{\mathrm{ev}}
=\operatorname{Sp}^*(J/\hspace{-1mm}/K). 
\]

Let $\mathsf{V}_{\mathbb{J}}=P(\mathbb{J})_1$ denote the vector space consisting of all 
odd primitive elements of $\mathbb{J}$, which is in fact a left $J$-module under the adjoint $J$-action.
Similarly we have the left $K$-module $\mathsf{V}_{\mathbb{K}}=P(\mathbb{K})_1$ consisting of all odd primitive elements
of $\mathbb{K}$, which is seen to be a $K$-submodule of $\mathsf{V}_{\mathbb{J}}$, so that we have the quotient
$K$-module 
$
\mathsf{Q}:=\mathsf{V}_{\mathbb{J}}/\mathsf{V}_{\mathbb{K}}
$
of $\mathsf{V}_{\mathbb{J}}$ by $\mathsf{V}_{\mathbb{K}}$. 
A main result of ours (Theorem \ref{T41} (1)) states that there is an isomprphism 
\begin{equation}\label{E11}
J\otimes_K\wedge(\mathsf{Q})\simeq \mathbb{J}/\hspace{-1mm}/\mathbb{K}
\end{equation}
of super-coalgebras, which turns into
\begin{equation}\label{E12}
J/\hspace{-1mm}/K\otimes \wedge(\mathsf{Q})\simeq \mathbb{J}/\hspace{-1mm}/\mathbb{K},
\end{equation}
in many special cases including the case where $\mathbb{J}$ is pointed; see Proposition \ref{P43}. 
Notice that the exterior algebra $\wedge(\mathsf{Q})$ is naturally a Hopf superalgebra in which every
element of $\mathsf{Q}$ is odd primitive, and it is regarded naturally as a left $K$-module super-coalgebra for  \eqref{E11}, and simply as a super-coalgebra for \eqref{E12}. 

A description of  $\mathbb{J}/\hspace{-1mm}/\mathbb{K}$ which looks similar to \eqref{E12} is given by
\cite[Proposition 3.14 (1)]{M2} under the assumption that $\mathbb{J}$ is irreducible, or namely,
the coradical of $\mathbb{J}$ is $\Bbbk 1$. But, our method of the proof is more conceptual, which uses the
notion of \emph{co-free super-coalgebras} (see Definition \ref{D32} for precise definition) and proves on the course 
that $J\otimes_K\wedge(\mathsf{Q})$ and $J/\hspace{-1mm}/K\otimes \wedge(\mathsf{Q})$ are co-free; see Proposition \ref{P42} and Example \ref{Ex32}. Furthermore, we obtain from \eqref{E11} 
the smoothness criteria (Theorem \ref{T42}):
(1)\ $\mathbb{J}/\hspace{-1mm}/\mathbb{K}$ is smooth if and only if so is $J/\hspace{-1mm}/K$, 
(2)\ The equivalent conditions are necessarily satisfied if $\operatorname{char}\Bbbk =0$. 

Throughout we will use Hopf-algebraic techniques; in particular, a key step to prove \eqref{E11} is 
to show a Maschke-type result such as found in \cite{D}; see Lemma \ref{L41}.  The techniques
are so useful that we will work mostly in the Hopf-algebra language, not in the group-theoretical one. 
They also make it easier to investigate the associated graded object. Indeed, we obtain a description 
of the graded coalgebra $\operatorname{gr}(\mathbb{J}/\hspace{-1mm}/\mathbb{K})$
associated to $\mathbb{J}/\hspace{-1mm}/\mathbb{K}$; see Theorem \ref{T41} (2). 

\subsection{Organization of the paper}
Section 2 is devoted to preliminaries, which contain basic definitions and results on supersymmetry
and Hopf superalgebras. 
Section 3 discusses co-free super-coalgebras. Section 4 presents main results and their proofs. 
Appendix discusses formal superschemes and supergroups from scratch, and translates some of the
results obtained in the text into the language of those formal super-objects. 

\section{Preliminaries}\label{Sec2}


\subsection{Supersymmetry}\label{Subsec2.1}
A \emph{super-vector space} is a vector space $V=V_{0} \oplus V_{1}$ (over $\Bbbk$) graded by the order-2 group $\mathbb{Z}_{2}=\{ 0, 1 \}$.
It is said to be \emph{purely even} (resp., \emph{purely odd}) if $V=V_{0}$ (resp., $V=V_{1}$).
The super-vector spaces $V, W, \dots$ form a monoidal category,
\[ \mathsf{SMod}_{\Bbbk} = (\mathsf{SMod}_{\Bbbk},\ \otimes,\ \Bbbk), \]
with respect to the natural tensor product $\otimes$ over $\Bbbk$, and the unit object $\Bbbk$ which is supposed to be purely even.
This monoidal category is symmetric with respect to the \emph{supersymmetry}
\begin{equation}\label{E21}
c_{V, W}: V \otimes W \overset{\simeq}{\longrightarrow} W \otimes V,\quad v \otimes w \mapsto (-1)^{|v| |w|} w \otimes v,
\end{equation}
where $v \in V$ and $w \in W$ are homogeneous elements with the parities (or degrees) $|v|, |w|$, respectively.

Objects, such as (co)algebra, Hopf algebra or Lie algebra, constructed in $\mathsf{SMod}_{\Bbbk}$ are called so as super-(co)algebra, Hopf superalgebra or Lie superalgebra, with ``super'' thus attached.
Ordinary objects, such as ordinary (co)algebra, \dots, are regarded as purely even super-objects.

If $\operatorname{char} \Bbbk = 3$, a Lie superalgebra $\mathfrak{g}$ is required to satisfy
\[ [v,[v,v]] = 0,\quad v \in \mathfrak{g}_{1}, \]
in addition to the alternativity and the Jacobi identity formulated in the super context.

Useful references for algebra and geometry in the supersymmetric context include Carmeli et al. \cite{CCF} and Manin \cite{Manin}. 

\subsection{Super-coalgebras and Hopf superalgebras}\label{Subsec2.2}
In what follows super-coalgebras are assumed to be super-cocommutative.
Accordingly, Hopf superalgebras are so, as well.

For a super-coalgebra $\mathbb{J}$, the structure is denoted by $\Delta_{\mathbb{J}}: \mathbb{J} \to \mathbb{J} \otimes \mathbb{J}, \varepsilon_{\mathbb{J}}: \mathbb{J} \to \Bbbk$, or simply by $\Delta, \varepsilon$.
The coproduct $\Delta$ is presented so as
\[ \Delta(a) = a_{(1)} \otimes a_{(2)},\quad a \in \mathbb{J} \]
by this variant of the Heynemann-Sweedler notation.
The super-cocommutativity assumption is presented by $\Delta = c_{\mathbb{J}, \mathbb{J}} \circ \Delta$.
In case $\mathbb{J}$ is a Hopf superalgebra, the antipode is denoted by $S_{\mathbb{J}}$ or $S$.

A \emph{graded coalgebra} is a non-negatively graded coalgebta $\mathcal{J} = \bigoplus_{n \geq 0} \mathcal{J}(n)$ which, regarded as a super-coalgebra $\mathcal{J} = \mathcal{J}_{0} \oplus \mathcal{J}_{1}$ with respect to the parity
\[
\mathcal{J}_{i} = \left\{
\begin{array}{cl}
\bigoplus\limits_{n~\mathrm{even}} \mathcal{J}(n), & i=0, \\
\bigoplus\limits_{n~\mathrm{odd}} \mathcal{J}(n), & i=1
\end{array}
\right.
\]
is super-cocommutative.
A \emph{graded Hopf superalgebra} is a graded super-coalgebra, equipped with a graded-algebra structure, which is a Hopf superalgebra with respect to the parity as above.

\subsection{The exterior algebra}\label{Subsec2.3}
Let $\mathsf{V}$ be a purely odd super-vector space.

The exterior algebra $\wedge(\mathsf{V}) = \bigoplus_{n \geq 0} \wedge^{n}(\mathsf{V})$ on $\mathsf{V}$ is a graded algebra, and is in fact a graded Hopf superalgebra in which every element $v \in \mathsf{V}\, (=\wedge^{1}(\mathsf{V}))$ is primitive, $\Delta(v) = 1\otimes v + v \otimes 1$; this is super-commutative as well as super-cocommutative.

In what follows, $\wedge(\mathsf{V})$ will be often regarded only as a graded coalgebra or a super-coalgebra.
To give an alternative description of that structure, first notice that given an integer $n > 1$, the symmetric group $\mathfrak{S}_{n}$ of degree $n$ acts on the $n$-fold tensor product $\mathsf{V}^{\otimes n}$ of $\mathsf{V}$, so that the $i$th fundamental transposition $(i, i+1)$ acts as
\[ \mathrm{id}_{\mathsf{V}}^{\otimes (i-1)} \otimes c_{\mathsf{V}, \mathsf{V}} \otimes \mathrm{id}_{\mathsf{V}}^{\otimes (n-i-1)}, \]
where $1 \leq i <n$.
We suppose that $\mathfrak{S}_n$ acts on the right, so that the action is given explicitly by
\[
(v_1\otimes v_2\otimes \cdots \otimes v_n)^{\sigma}=  \mathrm{sgn}(\sigma)\, 
v_{\sigma(1)}\otimes v_{\sigma(2)}\otimes\cdots \otimes v_{\sigma(n)},
\]
where $\sigma \in \mathfrak{S}_n$ and $v_i \in \mathsf{V}$, $1 \le i \le n$.
Let
\begin{equation}\label{E22}
A^{n}(\mathsf{V}):=(\mathsf{V}^{\otimes n})^{\mathfrak{S}_{n}}
\end{equation}
denote the subspace of $\mathsf{V}^{\otimes n}$ consisting of all $\mathfrak{S}_{n}$-invariants.
The \emph{symmetrizer}
\begin{equation}\label{E23}
\mathrm{sym}_{n}: \wedge^{n}(\mathsf{V}) \overset{\simeq}{\longrightarrow} A^{n}(\mathsf{V}),\quad v_{1} \wedge \cdots \wedge v_{n} \mapsto \sum_{\sigma \in \mathfrak{S}_n} (v_{1} \otimes \cdots \otimes v_{n})^{\sigma}
\end{equation}
gives a linear isomorphism; this is to be called the \emph{anti-symmetrizer} ordinarily in the non-super situation.
Let $A^{0}(\mathsf{V}) := \Bbbk\, (= \wedge^{0}(\mathsf{V}))$,\ $A^{1}(\mathsf{V}) := \mathsf{V}\, (= \wedge^{1}(\mathsf{V}))$, and set
\[ A(\mathsf{V}) := \bigoplus_{n \geq 0} A^{n}(\mathsf{V}). \]
For $n \geq 0$ and $0 \leq i \leq n$, we let
\begin{equation}\label{E23a}
\delta_{i, n-i}: A^{n}(\mathsf{V}) \to A^{i}(\mathsf{V}) \otimes A^{n-i}(\mathsf{V})
\end{equation}
denote the injection restricted from the canonical linear isomorphism $\mathsf{V}^{\otimes n} \overset{\simeq}{\longrightarrow} \mathsf{V}^{\otimes i} \otimes \mathsf{V}^{\otimes(n-i)}$.
The graded-coalgebra structure of $\wedge(\mathsf{V})$ is transferred through $\mathrm{sym}_{n}$ to $A(\mathsf{V})$
as follows: 
the coproduct is transferred to $\Delta: A(\mathsf{V}) \to A(\mathsf{V}) \otimes A(\mathsf{V})$ defined by
\[ \Delta(a) = \sum_{0 \leq i \leq n} \delta_{i, n-i}(a),\quad a \in A^{n}(\mathsf{V}),\ \, n \geq 0, \]
while the counit is to the projection $A(\mathsf{V}) \to A^{0}(\mathsf{V}) = \Bbbk$.
Here, as for $\Delta$ one should notice the formula
\[ 
\begin{split}
&\mathrm{sym}_{n}(v_{1} \wedge \cdots \wedge v_{n} ) =\\ 
&\sum_{\sigma} \mathrm{sgn}(\sigma)\, \mathrm{sym}_{i}(v_{\sigma(1)} \wedge \cdots \wedge v_{\sigma(i)}) \otimes \mathrm{sym}_{n-i}(v_{\sigma(i+1)} \wedge \cdots \wedge v_{\sigma(n)}), 
\end{split}
\]
where $\sigma$ runs over the $(i, n-i)$-shuffles, i.e., the permutations such that $\sigma(1) <\cdots< \sigma(i)$ 
and $\sigma(i+1) <\cdots< \sigma(n)$.
We can thus identify so as
\begin{equation}\label{E24}
\wedge(\mathsf{V}) = A(\mathsf{V})
\end{equation}
as graded coalgebras.

\subsection{The associated graded coalgebra}\label{Subsec2.4}
Let $\mathbb{J}$ be a super-coalgebra.
The pullback
\begin{equation}\label{E25}
J := \Delta^{-1} (\mathbb{J}_{0} \otimes \mathbb{J}_{0})
\end{equation}
of $\mathbb{J}_{0} \otimes \mathbb{J}_{0}\, (\subset \mathbb{J} \otimes \mathbb{J})$ along the coproduct is the largest purely even sub-super-coalgebra of $\mathbb{J}$; see \cite[Sect.~3]{M2}.
Let $F_{-1}\mathbb{J} = 0$,\ $F_{0} \mathbb{J} = J$.
For every integer $n>0$, let
\[ F_{n}\mathbb{J} := \mathrm{Ker}(\mathbb{J} \xrightarrow{\Delta_{n}} \mathbb{J}^{\otimes (n+1)} \rightarrow (\mathbb{J}/J)^{\otimes (n+1)}) \]
be the kernel of the composite of the $n$-iterated coproduct of $\mathbb{J}$ with the natural projection onto $(\mathbb{J}/J)^{\otimes (n+1)}$.
Then $\mathbb{J}$ is filtered (see \cite[Sect. 11.1]{Sw}), or more precisely, we have an ascending chain
\[ J = F_{0}\mathbb{J} \subset F_{1}\mathbb{J} \subset F_{2}\mathbb{J} \subset \cdots \]
of sub-super-coalgebras of $\mathbb{J}$, such that
\[ \mathbb{J} = \bigcup_{n \geq 0} F_{n}\mathbb{J},\quad \Delta(F_{n}\mathbb{J}) \subset \sum_{0 \leq i \leq n} F_{i}\mathbb{J} \otimes F_{n-i}\mathbb{J}. \]
The associated graded coalgebra
\[ \operatorname{gr}\mathbb{J} := \bigoplus_{n \geq 0}F_{n}\mathbb{J} / F_{n-1}\mathbb{J} \]
is indeed super-cocommutative.

Suppose that $\mathbb{J}$ is a Hopf superalgebra.
Then $J$ is a purely even Hopf superalgebra of $\mathbb{J}$, and we have $(F_{n}\mathbb{J})(F_{m}\mathbb{J}) \subset F_{n+m}\mathbb{J}$,\ $n,m \geq 0$.
Moreover, $\operatorname{gr}\mathbb{J}$ turns into a graded Hopf superalgebra.
The super-vector space
\begin{equation}\label{E26}
\operatorname{P}(\mathbb{J}) = \{\, v \in \mathbb{J} \mid \Delta(v) = 1 \otimes v + v \otimes 1 \, \}
\end{equation}
consisting of all primitive elements in $\mathbb{J}$ form a Lie superalgebra with respect to the super-commutator $[v,w] := vw - (-1)^{|v| |w|} wv$,
where $v$ and $w$ are homogeneous elements of 
$\operatorname{P}(\mathbb{J})$.
Define
\begin{equation}\label{E26a}
\mathsf{V}_{\mathbb{J}} := \operatorname{P}(\mathbb{J})_{1}, 
\end{equation}
the odd component of the Lie superalgebra.
This is stable under the right (resp., left) adjoint $J$-action
\begin{equation}\label{E27}
v \triangleleft a := S(a_{(1)}) v a_{(2)}\quad \text{ (resp., } a \triangleright v := a_{(1)} v S(a_{(2)})),
\end{equation}
where $a \in J$ and $v \in \mathsf{V}_{\mathbb{J}}$, whence it turns into a right (resp., left) purely odd $J$-supermodule.
Notice that $\mathbb{J}$ is a \emph{$J$-ring}, by which we mean an algebra equipped with an algebra map from $J$.
With an arbitrarily chosen, totally ordered basis $X=(x_{\lambda})_{\lambda \in \Lambda}$ of 
$\mathsf{V}_{\mathbb{J}}$,\
$\mathbb{J}$ is, as a $J$-ring, generated by $X$, and defined by the relations
\begin{equation}\label{E27a}
x_{\lambda}a = a_{(1)} (x_{\lambda} \triangleleft a_{(2)}),\quad x_{\lambda}x_{\mu} = -x_{\mu}x_{\lambda} + [x_{\lambda}, x_{\mu}], 
\end{equation}
where $a \in J$ and $\lambda > \mu$ in $\Lambda$.
Moreover, the linear map
\begin{equation}\label{E28}
\phi_{X}: J \otimes \wedge(\mathsf{V}_{\mathbb{J}}) \to \mathbb{J}
\end{equation}
defined by $\phi_{X}(a \otimes 1) = a$ and
\[ \phi_{X}(a \otimes (x_{\lambda_{1}} \wedge \cdots \wedge x_{\lambda_{n}})) = ax_{\lambda_{1}} \cdots x_{\lambda_{n}}, \]
where $n>0$ and $\lambda_{1} < \cdots < \lambda_{n}$ in $\Lambda$, is a left $J$-linear isomorphism of super-coalgebras; see \cite[Theorem 3.6]{M2}. In particular, $J \otimes \mathsf{V}_{\mathbb{J}}$ is naturally included in
$\mathbb{J}$,
\begin{equation}\label{E28a}
J \otimes \mathsf{V}_{\mathbb{J}}\subset \mathbb{J}. 
\end{equation}

\begin{rem}\label{R21}
Theorem 10 of \cite{M3} proves a category equivalence between (super-cocommutative) Hopf superalgebras
and dual Harish-Chandra pairs. We can make $\phi_X$ into a natural isomorphism of Hopf superalgebras, 
giving to $J \otimes \wedge(\mathsf{V}_{\mathbb{J}})$ some additional structures that are obtained from
the dual Harish-Chandra pair corresponding to $\mathbb{J}$. But, such a
more precise description of $\mathbb{J}$ will not be needed in the sequel. 
\end{rem}

Keep the notation as above. The category 
\[ \mathsf{SMod}_{J} = (\mathsf{SMod}_{J},\ \otimes,\ \Bbbk) \]
of right $J$-supermodules is a monoidal category, which is symmetric with respect to the supersymmetry \eqref{E21}
due to the cocommutativity of $J$.
The Hopf superalgebra $\wedge(\mathsf{V}_{\mathbb{J}})$ is in fact a Hopf algebra in $\mathsf{SMod}_{J}$ with respect to the $J$-action arising from the right adjoint action.
The associated smash product $J \ltimes \wedge(\mathsf{V}_{\mathbb{J}})$ is identified with $\operatorname{gr}\mathbb{J}$, since the second relation of \eqref{E27a} is reduced in $\operatorname{gr}\mathbb{J}$ to the super-commutativity $x_{\lambda}x_{\mu} = -x_{\mu}x_{\lambda}$.
To be more precise, $\mathsf{V}_{\mathbb{J}}$ is naturally isomorphic to $\operatorname{P}(\operatorname{gr}\mathbb{J})_{1}$, and the isomorphism together with the inclusion $J = (\operatorname{gr}\mathbb{J})(0) \hookrightarrow \operatorname{gr}\mathbb{J}$ uniquely extend to a canonical isomorphism
\begin{equation}\label{E29}
J \ltimes \wedge(\mathsf{V}_{\mathbb{J}}) \overset{\simeq}{\longrightarrow} \operatorname{gr}\mathbb{J}
\end{equation}
of graded Hopf superalgebras, which in fact coincides with $\operatorname{gr}(\phi_{X})$.

\subsection{Quick review of structure of Hopf superalgebras}\label{Subsec2.5}
Let $\mathbb{J}$ be a super-coalgebra. 
Every simple subcoalgebra of the $\mathbb{J}$, which is regarded as an ordinary coalgebra, is purely even.
The \emph{coradical} $\operatorname{Corad}\mathbb{J}$ of $\mathbb{J}$ is by definition the (necessarily, direct) sum of all simple subcoalgebras; it is seen to be included in the $J$ in \eqref{E25}.

Suppose that $\mathbb{J}$ is a Hopf superalgebra. 
We say that $\mathbb{J}$ is \emph{irreducible} if $\Bbbk 1$ is a unique simple subcoalgebra, or namely, $\Bbbk 1 = \operatorname{Corad}\mathbb{J}$.
It is known that $\mathbb{J}$ includes the largest irreducible Hopf sub-superalgebra denoted by $\mathbb{J}^{1}$.
If $\operatorname{char} \Bbbk = 0$, $\mathbb{J}^{1}$ coincides with the universal envelope $U(\operatorname{P}(\mathbb{J}))$ of the Lie superalgebra $\operatorname{P}(\mathbb{J})$ in \eqref{E26}.
We say that $\mathbb{J}$ is \emph{pointed} if the simple subcoalgebras of $\mathbb{J}$ are all 1-dimensional; this is the case if $\Bbbk$ is algebraically closed.
Suppose that $\mathbb{J}$ is pointed.
Then the group
\[ G(\mathbb{J}) = \{\, g \in \mathbb{J} \mid \Delta(g)=g \otimes g,\ \varepsilon(g)=1 \, \} \]
of the grouplike elements, which all are necessarily even for $\mathbb{J}$ being over a field, 
spans $\operatorname{Corad}\mathbb{J}$.
Moreover, $\mathbb{J}^{1}$ is stable under the conjugation by $G(\mathbb{J})$
\[ a \triangleleft g = g^{-1} a g,\quad a \in \mathbb{J}^{1},\ \, g \in G(\mathbb{J}), \]
and the resulting 
smash product $\Bbbk G(\mathbb{J}) \ltimes \mathbb{J}^{1}$
is naturally identified with $\mathbb{J}$.

\section{Co-free super-coalgebras}\label{Sec3}

Let $\mathbb{C}$ be a super-coalgebra; we remark this symbol shall not express the complex field, which does not appear in this paper.
Let $V$ be a super-vector space.
Due to the super-cocommutativity of $\mathbb{C}$, every left $\mathbb{C}$-super-comodule structure 
(or left $\mathbb{C}$-coaction) $\lambda: V \to \mathbb{C} \otimes V$  on $V$ can be identified with a right $\mathbb{C}$-super-comodule structure (or right $\mathbb{C}$-coaction) $\rho : V \to V \otimes \mathbb{C}$ through
\begin{equation}\label{E31}
\lambda = c_{V, \mathbb{C}} \circ \rho\ \, \text{ or } \ \, c_{\mathbb{C}, V} \circ \lambda = \rho,
\end{equation}
and vice versa.
Notice that $(V, \lambda, \rho)$ is then a $(\mathbb{C}, \mathbb{C})$-bi-super-comodule in the sense
\[ (\mathrm{id}_{\mathbb{C}} \otimes \rho) \circ \lambda =(\lambda \otimes \mathrm{id}_{\mathbb{C}}) \circ \rho. \]
Therefore, given a $\mathbb{C}$-super-comodule, we may regard it as any of left, right and bi-super-comodules, and do not specify which unless it is needed.

Given $\mathbb{C}$-super-comodules $V$ and $W$, the \emph{cotensor product} $V\, \square_{\mathbb{C}} W$ is defined by the equalizer diagram
\[
\xymatrix@C=15pt{
V\, \square_{\mathbb{C}} W \ar[r] & V \otimes W \ar@<0.5ex>[rr]^-{\rho_{V} \otimes \mathrm{id}_{W}} \ar@<-0.5ex>[rr]_-{\mathrm{id}_{V} \otimes \lambda_{W}} & & V \otimes \mathbb{C} \otimes W,
}
\]
where $\rho_{V}$ (resp., $\lambda_{V}$) denotes the right (resp., left) $\mathbb{C}$-coaction on $V$ (resp., $W$).
The left $\mathbb{C}$-coaction on $V$ and the right $\mathbb{C}$-coaction on $W$ give rise to the same (in the super sense as given by \eqref{E31}) coaction onto $V\, \square_{\mathbb{C}} W$, with which we regard 
$V\, \square_{\mathbb{C}} W$ as a $\mathbb{C}$-super-comodule.

We let $\mathsf{SMod}^{\mathbb{C}}$ denote the category of $\mathbb{C}$-super-comodules; we thus take ``right'' for this notation.
It is easy to see the following.
\begin{lemma}\label{L31}
$\mathsf{SMod}^{\mathbb{C}}$ forms a monoidal category,
\[ (\mathsf{SMod}^{\mathbb{C}},\ \square_{\mathbb{C}},\ \mathbb{C}) \]
whose tensor product is given by the cotensor product $\square_{\mathbb{C}}$, and whose unit object is $\mathbb{C}$.
This is in fact symmetric with respect to the supersymmetry
\[ {c_{V, W}|}_{V\, \square_{\mathbb{C}} W}: V\, \square_{\mathbb{C}} W \overset{\simeq}{\longrightarrow} W\, \square_{\mathbb{C}} V \]
restricted to the cotensor products.
\end{lemma}
\begin{definition}\label{D31}
A \emph{super-coalgebra over $\mathbb{C}$} or \emph{$\mathbb{C}$-super-coalgebra} is a pair $(\mathbb{D}, \omega)$ of a super-coalgebra $\mathbb{D}$ and a super-coalgebra map $\omega: \mathbb{D} \to \mathbb{C}$.
\end{definition}

Giving a $\mathbb{C}$-super-coalgebra is the same as giving a coalgebra in $(\mathsf{SMod}^{\mathbb{C}},\ \square_{\mathbb{C}},\ \mathbb{C})$.
To be more precise we have the following.
\begin{lemma}\label{L32}
If $(\mathbb{D}, \omega)$ is a $\mathbb{C}$-super-coalgebra, then $\mathbb{D}$, regarded as a $\mathbb{C}$-super-comodule with respect to $(\mathrm{id}_{\mathbb{D}} \otimes \omega) \circ \Delta_{\mathbb{D}}: \mathbb{D} \to \mathbb{D} \otimes \mathbb{C}$, is a coalgebra in $(\mathsf{SMod}^{\mathbb{C}},\ \square_{\mathbb{C}},\ \mathbb{C})$ 
whose coproduct is $\Delta_{\mathbb{D}}: \mathbb{D} \to \mathbb{D} \otimes \mathbb{D}$, regarded to be mapping into $\mathbb{D}\, \square_{\mathbb{C}} \mathbb{D}$, and whose counit  is $\omega: \mathbb{D} \to \mathbb{C}$.
Every coalgebra in the category arises uniquely in this way.
\end{lemma}

\begin{proof}
Given a coalgebra $\mathbb{D}$ in the category, its coproduct $\mathbb{D} \to \mathbb{D}\, \square_{\mathbb{C}} \mathbb{D} \hookrightarrow \mathbb{D} \otimes \mathbb{D}$ composed with the embedding into $\mathbb{D} \otimes \mathbb{D}$, and its counit $\mathbb{D} \to \mathbb{C} \xrightarrow{\varepsilon_{\mathbb{C}}} \Bbbk$ composed with the counit of $\mathbb{C}$ make $\mathbb{D}$ into a super-coalgebra.
This super-coalgebra, paired with the counit $\mathbb{D} \to \mathbb{C}$ in the category, is made into a $\mathbb{C}$-super-coalgebra.
This construction is seen to be an inverse procedure of the construction given in the lemma.
\end{proof}

Let $C$ be a (cocommutative) coalgebra, and let $\mathsf{V}$ be a purely odd $C$-super-comodule.
We emphasize that $C$ is supposed to be purely even, and $\mathsf{V}$ purely odd. 

\begin{definition}\label{D32}
A \emph{$C$-super-coalgebra on $\mathsf{V}$} is a  pair $(\mathbb{D}, \pi)$ of a $C$-super-coalgebra $\mathbb{D}$ and a $C$-super-colinear map $\pi: \mathbb{D} \to \mathsf{V}$. The $C$-super-coalgebras on $\mathsf{V}$
form a category, whose
morphisms are $C$-colinear super-coalgebra maps compatible with the maps to $\mathsf{V}$.
A $C$-super-coalgebra $(\mathbb{D}, \pi)$ on $\mathsf{V}$ is said to be \emph{co-free} if it is a terminal object of the category, or more explicitly, if given a $C$-super-coalgebra $(\mathbb{E}, \varpi)$ on $\mathsf{V}$, there exists a unique 
morphism $(\mathbb{E}, \varpi) \to (\mathbb{D}, \pi)$ of $C$-super-coalgebras on $\mathsf{V}$.
\end{definition}

\begin{prop}\label{P31}
There exists uniquely (up to isomorphism) a co-free $C$-super-coalgebra on $\mathsf{V}$.
\end{prop}
\begin{proof}
The uniqueness is obvious.
Let us construct explicitly a desired $(\mathbb{D}, \pi)$, which is in fact graded, $\mathbb{D}= \bigoplus_{n \geq 0} \mathbb{D}(n)$, as a coalgebra. We write simply $\square$ for $\square_{C}$.
Let
\[ \mathbb{D}(0):=C,\quad \mathbb{D}(1):=V. \]
Let $n >1$.
Notice that the cotensor product
\[ V^{\square n} = V~\square \cdots \square~V \]
of $n$ copies of $V$ is stable under the $\mathfrak{S}_{n}$-action.
Define
\[ \mathbb{D}(n):= (V^{\square n})^{\mathfrak{S}_{n}}\, (= V^{\square n} \cap A^{n}(V)) \]
to be the $C$-sub-comodule consisting of all $\mathfrak{S}_{n}$-invariants in $V^{\square n}$, which equals $V^{\square n} \cap A^{n}(V)$ in $V^{\otimes n}$.
We let $V^{\square 0} = C\, (=\mathbb{D}(0))$,\ $V^{\square 1}=V\, (= \mathbb{D}(1))$.

Let $n \geq 0$ and $0 \leq i \leq n$.
We see that the canonical $C$-super-colinear isomorphisms $V^{\square n} \overset{\simeq}{\longrightarrow} 
V^{\square i}~\square~V^{\square (n-i)}$ restrict to
\[ \mathbb{D}(n) \to \mathbb{D}(i)~\square~\mathbb{D}(n-i), \]
which we denote by $\Delta_{i, n-i}$; cf.~\eqref{E23a} in the case where $0 < i < n$.
Define a $C$-super-colinear map $\Delta: \mathbb{D} \to \mathbb{D}~\square~\mathbb{D}$ by
\[ \Delta(a) = \sum_{0 \leq i \leq n} \Delta_{i, n-i}(a),\quad a \in \mathbb{D}(n),\ n \geq 0. \]
Let $\varepsilon: \mathbb{D} \to \mathbb{D}(0)=C$ be the projection.
Then we see that $(\mathbb{D}, \Delta, \varepsilon)$ is a (graded) $C$-super-coalgebra; cf. the graded-coalgebra structure of $A(\mathsf{V})$ given in Section \ref{Subsec2.3}.

We wish to prove that this $\mathbb{D}$, paired with the projection $\pi: \mathbb{D} \to \mathbb{D}(1) = \mathsf{V}$, 
is co-free.
Given a $C$-super-coalgebra $(\mathbb{E}, \varpi)$ on $\mathsf{V}$, let us construct $C$-super-colinear maps $f_{n}: \mathbb{E} \to \mathbb{D}(n)$,\ $n \geq 0$.
Let
\[ f_{0}: \mathbb{E} \to \mathbb{D}(0) = C,\quad f_{1}: \mathbb{E} \to \mathbb{D}(1) = \mathsf{V} \]
be the super-coalgebra map equipped to $\mathbb{E}$, and be $\varpi$, respectively.
For $n>1$, define $f_{n}: \mathbb{E} \to \mathbb{D}(n)$ to be the composite
\begin{equation}\label{E32}
\mathbb{E} \xrightarrow{(\Delta_{\mathbb{E}})_{n-1}} \mathbb{E}^{\square n} \xrightarrow{\varpi^{\square n}} \mathsf{V}^{\square n},
\end{equation}
where $(\Delta_{\mathbb{E}})_{n-1}$ denote the $(n-1)$-iterated coproduct $\Delta_{\mathbb{E}}$ of $\mathbb{E}$.
This composite indeed maps into $\mathbb{D}(n)\, (=(\mathsf{V}^{\square n})^{\mathfrak{S}_{n}})$, since $\mathbb{E}$ is super-cocommutative.
Since the coradical $\operatorname{Corad} \mathbb{E}$ of $\mathbb{E}$ is purely even, and is, therefore, killed by $\varpi$, it follows that for every $a \in \mathbb{E}$, we have $f_{n}(a)=0$ for sufficiently large $n$.
Thus we can define $f: \mathbb{E} \to \mathbb{D}$ by
\[ f(a)=\sum_{n \geq 0} f_{n}(a),\quad a \in \mathbb{E}. \]

One sees that this $f$ is a unique morphism $\mathbb{E} \to \mathbb{D}$ of $C$-super-coalgebras on $\mathsf{V}$. 
Indeed, it is easy to see that $f$ is compatible with the counits.
Let $n>1$. 
To see that $f$ is compatible with the coproducts, one should use the formula
\[
\Delta_{i,n-i}\circ f_n=(f_i\, \square f_{n-i})\circ \Delta_{\mathbb{E}}, 
\]
where $0\leq i\leq n$. For the uniqueness one should use the fact that the composite
\[
\mathbb{D}\xrightarrow{\Delta_{n-1}} \mathbb{D}^{\square n}\xrightarrow{\pi^{\square n}} \mathsf{V}^{\square n}
=\mathbb{D}(n)
\]
coincides with the projection $\mathbb{D}\to \mathbb{D}(n)$. 
\end{proof}

We denote the thus constructed, co-free $C$-super-coalgebra on $\mathsf{V}$ by
\[ \operatorname{coF}^{C}(\mathsf{V}). \]
We emphasize that this is graded so that 
\begin{equation}\label{E33}
\operatorname{coF}^{C}(\mathsf{V})(0)=C,\quad \operatorname{coF}^{C}(\mathsf{V})(1)=\mathsf{V} 
\end{equation}
and the associated map to $\mathsf{V}$ is the projection onto the first component.
In addition, the largest purely even sub-super-coalgebra of $\operatorname{coF}^{C}(\mathsf{V})$
is $C$, as is seen from \eqref{E25} and the construction above. 

\begin{example}\label{Ex31}
Let us be in the special case where $C$ is the trivial coalgebra $\Bbbk$ spanned by a grouplike element, and $\mathsf{V}$ is, therefore, a purely odd super-vector space.
Then we have
\begin{equation}\label{E33a}
\operatorname{coF}^{\Bbbk}(\mathsf{V}) = \wedge(\mathsf{V})\, (=A(\mathsf{V}) \text{; see} \eqref{E24}). 
\end{equation}
\end{example}
\begin{rem}\label{R31}
This fact \eqref{E33a} is essentially shown in the second half of the proof of \cite[Propsition 3.11]{HMT} by the authors joint with Hoshi, which, however, contains an error; the composite given on Page 41, line --9, should read
\[ C \xrightarrow{\Delta_{n-1}} C^{\otimes n} \xrightarrow{\pi^{\otimes n}} \mathsf{U}_{1}^{\otimes n}, \]
which indeed maps into $A^{n}(\mathsf{U}_{1}) (= \wedge^{n}(\mathsf{U}_{1}))$; cf. \eqref{E32}.
\end{rem}
\begin{example}\label{Ex32}
Suppose that $C$ is an arbitrary (cocommutative) coalgebra, and $\mathsf{V}$ is co-free, or namely, $\mathsf{V} = C \otimes \mathsf{W}$, where $\mathsf{W}$ is a purely odd super-vector space.
We see from the construction in the last proof that
\begin{equation}\label{E33b}
\operatorname{coF}^{C}(C \otimes \mathsf{W}) = C \otimes \wedge(\mathsf{W}). 
\end{equation}
\end{example}

Notice that \eqref{E33a} and \eqref{E33b} are identifications of graded coalgebras, as well. 
The relevant following observation is simple, but will be used in the next section.

\begin{rem}\label{R32}
In general, given a graded coalgebra 
$\mathbb{D}= \bigoplus_{n \geq 0} \mathbb{D}(n)$, the neutral component $\mathbb{D}(0)$
is a (cocommutative) coalgebra and the first component $\mathbb{D}(1)$ is a purely odd $\mathbb{D}(0)$-super-comodule with respect to the relevant component of the coproduct 
\begin{equation}\label{E34}
\Delta_{\mathbb{D}}(1)_{1,0} : \mathbb{D}(1)\to \mathbb{D}(1)\otimes \mathbb{D}(0). 
\end{equation}
Therefore, $\mathbb{D}$, paired with the projection onto the neutral component, is a $\mathbb{D}(0)$-super-coalgebra,
which in turn, paired with the projection onto the
first component, is a $\mathbb{D}(0)$-super-coalgebra on $\mathbb{D}(1)$. 

Recall from \eqref{E33} that $\operatorname{coF}^{C}(\mathsf{V})$ is such a graded coalgebra $\mathbb{D}$ with the 
property 
\begin{equation}\label{E35}
\mathbb{D}(0)=C,\quad \mathbb{D}(1)=\mathsf{V}.
\end{equation}
It may be understood, as a $C$-super-coalgebra on $\mathsf{V}$, to be the one
which arises, as above, from its graded-coalgebra structure with the property \eqref{E35}. 
Moreover, it is a terminal object in the category of 
those graded coalgebras with that property; morphisms in the category are supposed to be identical in degrees
$0$ and $1$. 
\end{rem}

The notion of smooth (cocommutative) coalgebras (see \cite[p.1521, lines 5–7]{T0}, 
\cite[Definition 1.4, Proposition 1.5]{FS}) 
is directly generalized in the super context as follows. 

An inclusion $\mathbb{D}\hookrightarrow \mathbb{E}$ of super-coalgebras is said to be \emph{essential} 
if $\operatorname{Corad}\mathbb{E}\subset \mathbb{D}$; see \cite[Sect.~1]{MO}. 
A super-coalgebra $\mathbb{D}$ is said to be \emph{smooth}
if every essential inclusion $\mathbb{D}\hookrightarrow \mathbb{E}$ of $\mathbb{D}$ into another
super-coalgebra $\mathbb{E}$ splits. The condition is equivalent to saying that given an essential
inclusion $\mathbb{E}' \hookrightarrow \mathbb{E}$ of super-coalgebras, every super-coalgebra map 
$\mathbb{E}' \to \mathbb{D}$ extends to some super-coalgebra map $\mathbb{E} \to \mathbb{D}$.

A coalgebra, or a purely even super-coalgebra, is smooth as a coalgebra (in the sense of \cite{T0}, \cite{FS}) if and only if it is smooth as a super-coalgebra
(in the sense just defined). To see `only if', one should notice that every super-coalgebra map from a super-coalgebra, say, $\mathbb{E}$ to a coalgebra
uniquely factors through the
quotient purely even super-coalgebra $\mathbb{E}/\mathbb{E}_1\, (=\mathbb{E}_0)$ of $\mathbb{E}$. 

If $\operatorname{char}\Bbbk =0$, then every Hopf algebra is smooth, 
as is well known, and this fact,
combined with the isomorphism given by $\phi_X$ in \eqref{E28}, proves that every Hopf superalgebra is then smooth; see (the proof of) Theorem \ref{T42} (2). 

\begin{prop}\label{P32}
If $C$ is a smooth coalgebra and $\mathsf{V}$ is an injective purely odd $C$-super-comodule, then 
$\operatorname{coF}^C(\mathsf{V})$ is smooth.  
\end{prop}
\begin{proof}
Write $\mathbb{D}$ for $\operatorname{coF}^C(\mathsf{V})$. 
We are going to prove that an arbitrarily given essential inclusion 
$\mathbb{D}\hookrightarrow \mathbb{E}$ splits under the assumptions. 
By the smoothness assumption for $C$, the projection $\mathbb{D}\to \mathbb{D}(0)=C$ extends to a 
super-coalgebra map $\mathbb{E}\to C$, with which we regard $\mathbb{E}$ as a $C$-super-coalgebra. 
By the injectivity assumption for $\mathsf{V}$, the projection $\mathbb{D}\to \mathbb{D}(1)=\mathsf{V}$ extends to a
$C$-super-colinear map $\mathbb{E}\to \mathsf{V}$, which gives rise to a unique 
$C$-super-coalgebra map $\mathbb{E}\to \mathbb{D}$ on $\mathsf{V}$ 
by the co-freeness of $\mathbb{D}$. The map must be a retraction of the inclusion 
$\mathbb{D}\hookrightarrow \mathbb{E}$, regarded as a $C$-super-coalgebra map on $\mathsf{V}$,
again by the co-freeness. 
\end{proof}

\begin{rem}\label{R33}
If $\mathbb{D}$ is a smooth super-coalgebra, then the largest purely even sub-super-coalgebra, say, 
$D$ of $\mathbb{D}$
is smooth, as is easily seen using the fact that every super-coalgebra map from a purely even super-coalgebra to
$\mathbb{D}$ maps into $D$. In particular, if $\operatorname{coF}^C(\mathsf{V})$ is smooth, then $C$ is; see the
remark preceding Example \ref{Ex31}.
\end{rem}

\section{The quotient $\mathbb{J}/\hspace{-1mm}/\mathbb{K}$}\label{Sec4}

Let $\mathbb{J}$ be a Hopf superalgebra, and let 
$\mathbb{K}$ be a Hopf sub-superalgebra of $\mathbb{J}$.
Then the left super-ideal $\mathbb{J}\mathbb{K}^+$ of $\mathbb{J}$ generated by the augmentation super-ideal
$\mathbb{K}^+=\operatorname{Ker}(\varepsilon_{\mathbb{K}})$ of $\mathbb{K}$ is a super-coideal; see \cite[Proposition 1]{T} or \cite[p.286]{M2}.
We denote the resulting quotient super-coalgebra of $\mathbb{J}$ by 
\begin{equation}\label{E40a}
\mathbb{J}/\hspace{-1mm}/\mathbb{K}\, (=\mathbb{J}/\mathbb{J}\mathbb{K}^+).
\end{equation}
If $\mathbb{K}$ is a normal Hopf sub-superalgebra, or namely, if $\mathbb{J}\mathbb{K}^+=\mathbb{K}^+\mathbb{J}$
(or equivalently, $\mathbb{K}$ is stable under the adjoint $\mathbb{J}$-action, see \cite[Theorem 3.10 (3)]{M2}), then $\mathbb{J}/\hspace{-1mm}/\mathbb{K}$ is a quotient Hopf
superalgebra of $\mathbb{J}$, in which case the structure of $\mathbb{J}/\hspace{-1mm}/\mathbb{K}$ is rather clear
from the results obtained in \cite[Sect. 3]{M3}. So, in what follows, main interest of ours will be in the case
where $\mathbb{K}$ is not normal. 

We are going to investigate the structures of this super-coalgebra 
and of the associated graded coalgebra $\operatorname{gr}(\mathbb{J}/\hspace{-1mm}/\mathbb{K})$;
the results will be then applied to show smoothness criteria for $\mathbb{J}/\hspace{-1mm}/\mathbb{K}$. 
One will see easily that parallel results hold for the objects 
$\mathbb{J}\backslash\hspace{-1mm}\backslash\mathbb{K}\, (=\mathbb{J}/\mathbb{K}^+\mathbb{J})$,\
$\operatorname{gr}(\mathbb{J}\backslash\hspace{-1mm}\backslash\mathbb{K})$ analogously constructed
with the side switched. 

Let $J$ and $K$ be the largest purely even Hopf sub-superalgebras of $\mathbb{J}$ and of $\mathbb{K}$
(see Section \ref{Subsec2.4}), and set
\[ \mathsf{V}_{\mathbb{J}} :=\operatorname{P}(\mathbb{J})_{1}, \quad \mathsf{V}_{\mathbb{K}}:=\operatorname{P}(\mathbb{K})_{1} \]
as in \eqref{E26a}.
These are purely odd right (and left) supermodules over $J$ and over $K$, respectively; see \eqref{E27}. 
For a while we regard these as ordinary modules, keeping their parity in mind.
The natural maps $K \to J$ and $\mathsf{V}_{\mathbb{K}} \to \mathsf{V}_{\mathbb{J}}$ induced from the inclusion $\mathbb{K} \hookrightarrow \mathbb{J}$ are injections, 
through which we can regard $K$ as a Hopf subalgebra of $J$, and $\mathsf{V}_{\mathbb{K}}$ as a $K$-submodule of $\mathsf{V}_{\mathbb{J}}$.

Define
\begin{equation}\label{E40b}
\mathsf{Q} := \mathsf{V}_{\mathbb{J}} / \mathsf{V}_{\mathbb{K}}. 
\end{equation}
We let $\mathsf{Mod}_{K}$ denote the category of right $K$-modules, which is in fact a symmetric 
monoidal category, $(\mathsf{Mod}_{K}, \otimes, K)$. Since $J$ is a coalgebra in the category, we have
the category
\begin{equation}\label{E40c}
\mathsf{Mod}_{K}^{J}\, (= (\mathsf{Mod}_{K})^{J}) 
\end{equation}
of right $J$-comodules in $\mathsf{Mod}_{K}$. 
We have the short exact sequence
\[ 0 \to \mathsf{V}_{\mathbb{K}} \to \mathsf{V}_{\mathbb{J}} \to \mathsf{Q} \to 0 \]
in $\mathsf{Mod}_{K}$. Tensored with $J$, this gives rise to the short exact sequence
\begin{equation}\label{E41}
0 \to J \otimes \mathsf{V}_{\mathbb{K}} \to J \otimes \mathsf{V}_{\mathbb{J}} \to J \otimes \mathsf{Q} \to 0
\end{equation}
in $\mathsf{Mod}_{K}^{J}$, in which
$K$ acts diagonally on each tensor product, while $J$ coacts on the single tensor factor $J$.

\begin{lemma}\label{L41}
The short exact sequence \eqref{E41} in $\mathsf{Mod}_{K}^{J}$ splits.
\end{lemma}
\begin{proof}
Clearly, the surjection $J \otimes \mathsf{V}_{\mathbb{J}} \to J \otimes \mathsf{Q}$ in \eqref{E41} splits $J$-colinearly.
Therefore, it splits, as desired, as a morphism in $\mathsf{Mod}_{K}^{J}$, as is seen from the following.
\medskip

\noindent
\textbf{Fact.}
\emph{A surjective morphism $p:\mathsf{M} \to \mathsf{N}$ in $\mathsf{Mod}_{K}^{J}$ splits if it splits $J$-colinearly.}
\medskip

This fact is proved by dualizing the argument proving \cite[Theorem 1]{D}, as follows.
Since $J$ is projective as a right $K$-module by \cite[Theorem 1.3]{M}, we have a right $K$-linear map 
$\xi: J \to K$ such that $\varepsilon_{J} =\varepsilon_{K} \circ \xi$.
Let $n \mapsto n^{(0)}\otimes n^{(1)}$ represent the the $J$-coaction $\rho_{\mathsf{N}} : \mathsf{N} \to \mathsf{N} \otimes J$ 
on $\mathsf{N}$.
By direct computations we see that 
the $K$-action $\mathsf{N} \otimes K \to \mathsf{N},\ n \otimes c \mapsto nc$ splits in $\mathsf{Mod}_{K}^{J}$, having
\[ \sigma: \mathsf{N} \to \mathsf{N} \otimes K,\ \ \sigma(n) = n^{(0)} S_{K}(\xi(n^{(1)})_{(1)}) \otimes \xi(n^{(1)})_{(2)} \]
as a section. To see, for example, that $\sigma$ is $J$-colinear, we compute the $J$-coaction on $\sigma(n)$ so that
\[
\begin{split}
\rho_{\mathsf{N}}(\sigma(n))&=n^{(0)} S_{K}(\xi(n^{(2)})_{(2)}) \otimes \xi(n^{(2)})_{(3)}\otimes n^{(1)}S_{K}(\xi(n^{(2)})_{(1)})\xi(n^{(2)})_{(4)}\\
&=n^{(0)} S_{K}(\xi(n^{(2)})_{(1)}) \otimes \xi(n^{(2)})_{(4)}\otimes n^{(1)} S_{K}(\xi(n^{(2)})_{(2)})\xi(n^{(2)})_{(3)}\\
&=n^{(0)} S_{K}(\xi(n^{(2)})_{(1)}) \otimes \xi(n^{(2)})_{(2)}\otimes n^{(1)}\\
&=\sigma(n^{(0)})\otimes n^{(1)},
\end{split}
\]
where $n^{(0)}\otimes n^{(1)}\otimes n^{(2)}$ represents $\rho_{\mathsf{N}}(n^{(0)})\otimes n^{(1)}\, 
(=n^{(0)}\otimes\Delta_J(n^{(1)}))$; notice that 
the second and the last equalities hold since $K$ and $J$ are cocommutative. 
Choose a $J$-colinear section $s: \mathsf{N} \to \mathsf{M}$ of $p$.
Then the composite
\[ \mathsf{N} \overset{\sigma}{\longrightarrow} \mathsf{N} \otimes K \xrightarrow{s \otimes \mathrm{id}_{K}} \mathsf{M} \otimes K \longrightarrow \mathsf{M}, \]
where the last arrow indicates the $K$-action on $\mathsf{M}$, is seen to be a desired section.
\end{proof}

By Lemma \ref{L41} we can choose a section $\gamma: J \otimes \mathsf{Q} \to J \otimes \mathsf{V}_{\mathbb{J}}$ of the natural surjective morphism $J \otimes \mathsf{V}_{\mathbb{J}} \to J \otimes \mathsf{Q}$ in $\mathsf{Mod}_{K}^{J}$.
Notice that it is necessarily of the form
\[ \gamma(a \otimes q) = a_{(1)} \otimes g(a_{(2)} \otimes q), \]
where $a \in J$, $q \in \mathsf{Q}$ and $g = (\varepsilon_{J} \otimes \mathrm{id}_{\mathsf{Q}}) \circ \gamma$.
Recall here that we are discussing purely odd super-objects.

In addition, recall from Example \ref{Ex32} that 
$J \otimes \wedge(\mathsf{Q})$ and $J \otimes \wedge(\mathsf{V}_{\mathbb{J}})$ 
are the co-free $J$-super-coalgebras on the purely odd $J$-super-comodules $J\otimes \mathsf{Q}$ and
$J\otimes \mathsf{V}_{\mathbb{J}}$, respectively. 
Then we see that the $J$-colinear map $\gamma$ between those $J$-super-comodules uniquely extends to a graded coalgebra map
\[ \widetilde{\gamma}: J \otimes \wedge(\mathsf{Q}) \to J \otimes \wedge(\mathsf{V}_{\mathbb{J}}), \]
such that
\[ \widetilde{\gamma}(0) = \mathrm{id}_{J},\quad \widetilde{\gamma}(1) =\gamma \]
in degrees 0 and 1, and
\[ \widetilde{\gamma}(n)(a \otimes (q_{1} \wedge \cdots \wedge q_{n})) = a_{(1)} \otimes (g(a_{(2)} \otimes q_{1}) \wedge \cdots \wedge g(a_{(n+1)} \otimes q_{n})) \]
in degree $n \geq 2$, where $a \in J$ and $q_{i} \in \mathsf{Q}, 1 \leq i \leq n$.
As a super analogue of \eqref{E40c} we have the category $\mathsf{SMod}_{K}^{J}$ of $J$-comodules
in $\mathsf{SMod}_{K}$, which is indeed a symmetric monoidal category,
\begin{equation}\label{E41a}
(\mathsf{SMod}_{K}^{J},\ \square_{J},\ J),
\end{equation}
just as $(\mathsf{SMod}^{J}, \square_J, J)$ is; see Lemma \ref{L31}. 
Notice that $J \otimes \wedge(\mathsf{Q})$ and $J \otimes \wedge(\mathsf{V}_{\mathbb{J}})$ are
co-free coalgebras in $(\mathsf{SMod}_{K}^{J},\ \square_{J},\ J)$ in a generalized sense 
(defined in an obvious manner), and 
$\widetilde{\gamma}$, arising from the morphism $\gamma$ in the category, is a coalgebra morphism in the category. 

Define a morphism in $\mathsf{SMod}_{K}^{J}$ by
\begin{equation}\label{E41b}
\tau: J \otimes \wedge(\mathsf{V}_{\mathbb{J}}) \to \wedge(\mathsf{V}_{\mathbb{J}}) \otimes J,\quad \tau(a \otimes u) = (u  \triangleleft S(a_{(1)})) \otimes a_{(2)},
\end{equation}
where $a \in J$, $u \in \wedge(\mathsf{V}_{\mathbb{J}}).$
Here, on the target $\wedge(\mathsf{V}_{\mathbb{J}}) \otimes J$, $K$ acts and $J$ coacts on the tensor factor $J$.
We see that $\tau$ is in fact a graded coalgebra isomorphism with inverse
\[ \wedge(\mathsf{V}_{\mathbb{J}})\otimes J\to  J \otimes \wedge(\mathsf{V}_{\mathbb{J}}),
\quad u \otimes a \mapsto a_{(1)} \otimes (u \triangleleft a_{(2)}). \]
Choosing a totally ordered basis $X=(x_{\lambda})_{\lambda \in \Lambda}$ of $\mathsf{V}_{\mathbb{J}}$, 
one obtains the opposite-sided version $\phi'_{X}: \wedge(\mathsf{V}_{\mathbb{J}}) \otimes J \to 
\mathbb{J}$ of the $\phi_X$ in \eqref{E28}, which is a coalgebra isomorphism in $\mathsf{SMod}_{J}$.
Define $\Gamma: J \otimes \wedge(\mathsf{Q}) \to \mathbb{J}$ to be the composite
\[ J \otimes \wedge(\mathsf{Q}) \overset{\widetilde{\gamma}}{\longrightarrow} J \otimes \wedge(\mathsf{V}_{\mathbb{J}}) \overset{\tau}{\longrightarrow} \wedge(\mathsf{V}_{\mathbb{J}}) \otimes J \xrightarrow{\phi'_{X}} \mathbb{J}. \]
This is a coalgebra morphism in $\mathsf{SMod}_{K}$, which sends $a \otimes 1_{\wedge(\mathsf{Q})}$ to $a$ for every $a \in J$.

\begin{prop}\label{P41}
The right $\mathbb{K}$-linear extension of $\Gamma$
\[ \Gamma^{\mathbb{K}}: (J \otimes \wedge(\mathsf{Q}))\otimes_{K} \mathbb{K} \to \mathbb{J},\quad \Gamma^{\mathbb{K}}(w \otimes_{K} z ) = \Gamma(w)z, \]
where $w \in J \otimes \wedge(\mathsf{Q})$,\ $z \in \mathbb{K}$, is a coalgebra isomorphism in $\mathsf{SMod}_{\mathbb{K}}$, which sends $(a \otimes 1_{\wedge(\mathsf{Q})}) \otimes_{K} 1_{\mathbb{K}}$ to 
$a$ for every $a \in J$.
\end{prop}
\begin{proof}
Clearly, $\Gamma^{\mathbb{K}}$ is a coalgebra morphism in $\mathsf{SMod}_{\mathbb{K}}$, which sends $(a \otimes 1_{\wedge(\mathsf{Q})}) \otimes_{K} 1_{\mathbb{K}}$ to $a$
for every $a \in J$.
To see that it is an isomorphism, we aim to prove that $\operatorname{gr}(\Gamma^{\mathbb{K}})$ is bijective.
Notice that
\[ \operatorname{gr}(\phi'_{X} \circ \tau) = \operatorname{gr}(\phi'_{X}) \circ \tau : J\otimes \wedge(\mathsf{V}_{\mathbb{J}}) \to \operatorname{gr}\mathbb{J}  \]
coincides with the canonical isomorphism
$J \ltimes \wedge(\mathsf{V}_{\mathbb{J}})\overset{\simeq}{\longrightarrow}\operatorname{gr}\mathbb{J}$ 
in \eqref{E29}.
Using the similar isomorphism 
$K \ltimes \wedge(\mathsf{V}_{\mathbb{K}})\overset{\simeq}{\longrightarrow}\operatorname{gr}\mathbb{K}$,
one sees that $\operatorname{gr}(\Gamma^{\mathbb{K}})$ coincides with the composite
\[ J \otimes \wedge(\mathsf{Q}) \otimes \wedge(\mathsf{V}_{\mathbb{K}}) \xrightarrow{\widetilde{\gamma}\otimes \mathrm{id}_{\wedge(\mathsf{V}_{\mathbb{K}})}} J \otimes \wedge(\mathsf{V_{\mathbb{J}}}) \otimes \wedge(\mathsf{V}_{\mathbb{K}}) \xrightarrow{\mathrm{id}_{J} \otimes \mathrm{prod}} J \otimes \wedge(\mathsf{V_{\mathbb{J}}}), \]
where $\mathrm{prod}: \wedge(\mathsf{V}_{\mathbb{J}}) \otimes \wedge(\mathsf{V}_{\mathbb{K}}) \to \wedge(\mathsf{V}_{\mathbb{J}})$ denotes the product in $\wedge(\mathsf{V}_{\mathbb{J}})$.

Regard $\wedge(\mathsf{V}_{\mathbb{K}})\subset \wedge(\mathsf{V}_{\mathbb{J}})$
as Hopf superalgebras, naturally
as in the beginning of Section \ref{Subsec2.3}. 
Clearly, $\operatorname{gr}(\Gamma^{\mathbb{K}})$ is a coalgebra morphism in $\mathsf{SMod}_{\wedge(\mathsf{V}_{\mathbb{K}})}$.
Since $\varepsilon \otimes \mathrm{id}_{\wedge(\mathsf{V}_{\mathbb{K}})} : J \otimes \wedge(\mathsf{V}_{\mathbb{J}}) \to \wedge(\mathsf{V}_{\mathbb{J}})$ is a coalgebra morphism in $\mathsf{SMod}_{\wedge(\mathsf{V}_{\mathbb{K}})}$,
we can regard $\operatorname{gr}(\Gamma^{\mathbb{K}})$ as a $\wedge(\mathsf{V}_{\mathbb{J}})$-coalgebra morphism 
in $\mathsf{SMod}_{\wedge(\mathsf{V}_{\mathbb{K}})}$ (in a sense slightly generalized from Definition \ref{D31}), 
or in other words (see Lemma \ref{L32}), as a coalgebra morphism in 
\[ (\mathsf{SMod}_{\wedge(\mathsf{V}_{\mathbb{K}})}^{\wedge(\mathsf{V}_{\mathbb{J}})},\ \square_{\wedge(\mathsf{V}_{\mathbb{J}})},\ \wedge(\mathsf{V}_{\mathbb{J}})), \]
which is indeed a symmetric monoidal category just as the one in \eqref{E41a}. 
Since $\wedge(\mathsf{V}_{\mathbb{J}})/\hspace{-1mm}/\wedge\! (\mathsf{V}_{\mathbb{K}})$
is naturally isomorphic to $\wedge(\mathsf{Q})$,  it follows essentially by \cite[Theorem 1]{T} (see also \cite[Proposition 1.1]{M2}) 
that the symmetric monoidal category above is equivalent to
\[ (\mathsf{SMod}^{\wedge(\mathsf{Q})},\ \square_{\wedge(\mathsf{Q})},\ \wedge(\mathsf{Q})) \]
through the functor which assigns to an object $M$ in the former category,
\[ M \otimes_{\wedge(\mathsf{V}_{\mathsf{K}})} \Bbbk \, (= M/M(\wedge(\mathsf{V}_{\mathbb{K}}))^+),  \]
where one should notice $(\wedge(\mathsf{V}_{\mathbb{K}}))^+=\bigoplus_{n>0}\wedge^n(\mathsf{V}_{\mathbb{K}})$.
Therefore, for our aim, it suffices to prove that
\[ \operatorname{gr}(\Gamma^{\mathbb{K}}) \otimes_{\wedge(\mathsf{V}_{\mathbb{K}})} \Bbbk: J \otimes \wedge(\mathsf{Q}) \to J \otimes \wedge(\mathsf{V}_{\mathbb{J}}) \otimes_{\wedge(\mathsf{V}_{\mathbb{K}})} \Bbbk = J \otimes \wedge(\mathsf{Q}) \]
is bijective, or is indeed the identity map.
One sees that this is a graded coalgebra endomorphism which is identical clearly in degree $0$, and also in degree $1$
by choice of $\gamma$. 
By co-freeness of $J \otimes \wedge(\mathsf{Q})$ (see Example \ref{Ex32}) it must be the identity map 
in view of Remark \ref{R32}. 
\end{proof}

Now, we regard $\mathsf{V}_{\mathbb{K}} \subset \mathsf{V}_{\mathbb{J}}$ as left $K$-modules with respect to the left adjoint action (see \eqref{E27}), 
so that $\mathsf{Q}$ is a left $K$-module, and $\wedge(\mathsf{Q})$ is a coalgebra in the monoidal
category $_{K}\mathsf{SMod}$ of left $K$-supermodules.
Construct $J \otimes_{K} \wedge(\mathsf{Q})$.
This is a graded coalgebra such that 
\begin{equation}\label{E41c}
(J \otimes_{K} \wedge(\mathsf{Q}))(0)=J/\hspace{-1mm}/K\, (:=J/JK^+),\quad
(J \otimes_{K} \wedge(\mathsf{Q}))(1)=J \otimes_{K}\mathsf{Q},
\end{equation}
whence it is a $J/\hspace{-1mm}/K$-super-coalgebra with respect to the projection onto the neutral component;
notice that the associated $J/\hspace{-1mm}/K$-coaction on $J \otimes_{K}\mathsf{Q}$ (see \eqref{E34})
is the natural one on the tensor factor $J$. 
Moreover, we have the following.

\begin{prop}\label{P42}
$J \otimes_{K} \wedge(\mathsf{Q})$ is co-free on $J \otimes_{K} \mathsf{Q}$.
\end{prop}

This proposition and the following will be proved below together.

\begin{prop}\label{P43}
Assume one of the following (i)--(iv):
\begin{itemize}
\item[(i)] $J$ is pointed (this is satisfied if $\Bbbk$ is algebraically closed);
\item[(ii)] $K$ is finite-dimensional;
\item[(iii)] $\mathsf{V}_{\mathbb{K}}$ is stable in $\mathsf{V}_{\mathbb{J}}$ under the $J$-action (this is satisfied if $\mathbb{K}$ is normal in $\mathbb{J}$);
\item[(iv)] 
\begin{itemize}
\item[(a)] $K$ is smooth as a coalgebra (this is satisfied if $\operatorname{char}\Bbbk = 0$), and 
\item[(b)] $K$ includes the coradical $\operatorname{Corad} J$ of $J$.
\end{itemize}
\end{itemize}
Then there is a $J/\hspace{-1mm}/K$-colinear isomorphism $J \otimes_{K} \mathsf{Q}\simeq J/\hspace{-1mm}/K \otimes \mathsf{Q}$. It extends to an isomorphism of graded coalgebras
\begin{equation}\label{E42}
J \otimes_{K} \wedge(\mathsf{Q}) \simeq J/\hspace{-1mm}/K \otimes \wedge(\mathsf{Q})
\end{equation}
which is the identity map of $J/\hspace{-1mm}/K$ in degree $0$, and is the 
previous isomorphism in degree $1$.
\end{prop}

\begin{proof}[Proof of Propositions \ref{P42} and \ref{P43}]
By \eqref{E41c}
we have a canonical morphism
\begin{equation}\label{E43}
J \otimes_{K} \wedge(\mathsf{Q}) \to \mathrm{coF}^{J/\hspace{-1mm}/K}(J \otimes_{K} \mathsf{Q})
\end{equation}
of graded coalgebras, which is identical in degrees 0 and 1.
We wish to show that this is an isomorphism, which will prove Proposition \ref{P42}.
Replacing everything with its base extension to the algebraic closure $\overline{\Bbbk}$ of $\Bbbk$, 
we may suppose that $\Bbbk$ is algebraically closed.
Then $J$ is pointed, whence we have a $J/\hspace{-1mm}/K$-colinear and right $K$-linear isomorphism
\begin{equation}\label{E44}
J \simeq J/\hspace{-1mm}/K \otimes K 
\end{equation}
by \cite[Theorem 1.3 (4)]{M}. 
With $\otimes_K\mathsf{Q}$ and $\otimes_K \! \wedge\! (\mathsf{Q})$ applied, this induces
a $J/\hspace{-1mm}/K$-colinear isomorphism 
$J \otimes_{K} \mathsf{Q}\simeq J/\hspace{-1mm}/K \otimes \mathsf{Q}$ and an isomorphism
such as \eqref{E42}.
The morphism \eqref{E43} is then identified, through the isomorphisms just obtained, with a graded coalgebra map
\[ J/\hspace{-1mm}/K \otimes \wedge(\mathsf{Q}) \to \mathrm{coF}^{J/\hspace{-1mm}/K}(J/\hspace{-1mm}/K \otimes \mathsf{Q}) \]
which is identical in degrees 0 and 1.
By co-freeness of $J/\hspace{-1mm}/K \otimes \wedge(\mathsf{Q})$ (see Example \ref{Ex32})
this must be an isomorphism in view of Remark \ref{R32}. 

Let us turn to Proposition \ref{P43}. 
The argument above proves it in Cases (i) and (ii), since we then have an isomorphism such as 
\eqref{E44} by \cite[Theorem 1.3 (4)]{M}, again.
In Case (iv), as well, we have such an isomorphism. To see this, assume (iv). 
By modifying the proof of \cite[Theorem 4.1, Lemma 4.2]{M1} into the cocommutative situation, we
see that the inclusion $K \hookrightarrow J$ has a right $K$-linear coalgebra retraction, say, $r : J \to K$,
and
\[ J \to J/\hspace{-1mm}/K \otimes K,\quad a \mapsto \overline{a}_{(1)} \otimes r(a_{(2)}) \]
gives a desired isomorphism. Here and below we let $a \mapsto \overline{a}$ present 
the natural projection $J \to J/\hspace{-1mm}/K$.

Finally, assume (iii). Then $\mathsf{Q}$ is naturally a left $J$-module. 
We see that
\begin{equation}\label{E45}
J \otimes_{K} \mathsf{Q} \to J/\hspace{-1mm}/K \otimes \mathsf{Q},\quad a \otimes_{K} q \mapsto \overline{a}_{(1)} \otimes (a_{(2)} \triangleright q)
\end{equation}
is a $J/\hspace{-1mm}/K$-colinear isomorphism, which indeed has $\overline{a} \otimes q \mapsto a_{(1)} \otimes_{K} (S(a_{(2)}) \triangleright q)$ as an inverse; see \cite[p.456, line 6]{T} or \cite[Eq. (4)]{M1}. 
An obvious modification with $\mathsf{Q}$ replaced by $\wedge(\mathsf{Q})$
gives an isomorphism such as \eqref{E42}. Alternatively, the isomorphism follows from \eqref{E45}, in virtue of
the co-freeness shown by Proposition \ref{P42} and Example \ref{E32}. 
\end{proof}

\begin{rem}
It is known that there exist (cocommutative) Hopf algebras $J\supset K$ over a non-algebraically closed field,
for which there does not exist any 
isomorphism such as \eqref{E44}, or $J$ is not even free as a left or right $K$-module; see \cite[Sect. 5]{T}, for example.
\end{rem}

\begin{theorem}\label{T41}
We have the following.
\begin{itemize}
\item[(1)] There is an isomorphism
\[ J \otimes_{K} \wedge(\mathsf{Q}) \overset{\simeq}{\longrightarrow} \mathbb{J}/\hspace{-1mm}/\mathbb{K} \]
of super-coalgebras, which, restricted to
\[ J/\hspace{-1mm}/K\, (= J \otimes_{K} \wedge^{0}(\mathsf{Q}))\quad \text{and}\quad J \otimes_{K} \mathsf{Q}\, (= J \otimes_{K} \wedge^{1}(\mathsf{Q})), \]
is natural in the sense that it is induced from the natural projection $\mathbb{J} \to \mathbb{J}/\hspace{-1mm}/\mathbb{K}$ restricted to
\[ J\quad \text{and}\quad J \otimes \mathsf{V}_{\mathbb{J}}\ \, (\text{see \eqref{E28a}}). \]
\item[(2)] $\operatorname{gr}\mathbb{K}$ is naturally regarded as a graded Hopf sub-superalgebra of $\operatorname{gr}\mathbb{J}$.
The natural projection $\operatorname{gr}\mathbb{J} \to \operatorname{gr}(\mathbb{J}/\hspace{-1mm}/\mathbb{K})$ induces an isomorphism
\begin{equation}\label{E46}
\operatorname{gr}\mathbb{J}/\hspace{-1mm}/\operatorname{gr}\mathbb{K} \overset{\simeq}{\longrightarrow} \operatorname{gr}(\mathbb{J}/\hspace{-1mm}/\mathbb{K})
\end{equation}
of graded coalgebras.
These graded coalgebras are naturally isomorphic to the co-free $J/\hspace{-1mm}/K$-super-coalgebra 
$J \otimes_{K} \wedge(\mathsf{Q})$ on $J \otimes_{K} \mathsf{Q}$; see Proposition \ref{P42}.
\end{itemize}
\end{theorem}
\begin{proof}
(1) Recall from Proposition \ref{P41} the isomorphism $\Gamma^{\mathbb{K}}: (J \otimes \wedge(\mathsf{Q})) \otimes_{K} \mathbb{K} \overset{\simeq}{\longrightarrow} \mathbb{J}$, which has the property, among others, that it sends $(a \otimes 1_{\wedge(\mathsf{Q})}) \otimes_{K} 1_{\mathbb{K}}$ to $a$ for every $a \in J$.
With $\otimes_{K} \mathbb{K}/\mathbb{K}^{+}$ applied, it induces a desired isomorphism.
To verify the prescribed naturality, one uses the property above and the fact that the composite 
\[
J \otimes \mathsf{V}_{\mathbb{J}} \xrightarrow{\tau(1)} \mathsf{V}_{\mathbb{J}}\otimes J 
\xrightarrow{\mathrm{product}}\mathbb{J}
\]
of the first component $\tau(1)$ of the $\tau$ (see \eqref{E41b}) with the product in $\mathbb{J}$ coincides with the product map
$J \otimes \mathsf{V}_{\mathbb{J}}\to \mathbb{J}$, $a \otimes v\mapsto av$. 

(2) We have the following commutative diagram in which the bottom horizontal arrow indicates the isomorphism just obtained,
\[
\begin{xy}
(0,0)   *++{(J \otimes \wedge(\mathsf{Q})) \otimes_K \mathbb{K}}  ="1",
(40,0)  *++{\mathbb{J}}    ="2",
(0,-18) *++{J \otimes_K \wedge(\mathsf{Q})} ="3",
(40,-18)*++{\mathbb{J} /\hspace{-1mm}/\mathbb{K}.} ="4",
{"1" \SelectTips{cm}{} \ar @{->}_-{\Gamma^{\mathbb{K}}}^-{\simeq} "2"},
{"1" \SelectTips{cm}{} \ar @{->>}_{\otimes_{\mathbb{K}} \mathbb{K}/\mathbb{K}^{+}} "3"},
{"2" \SelectTips{cm}{} \ar @{->>}^-{\mathrm{projection}} "4"},
{"3" \SelectTips{cm}{} \ar @{->}^-{\simeq} "4"},
\end{xy}
\]
This, with $\operatorname{gr}$ applied, induces
\[
\begin{xy}
(0,0)   *++{(J \otimes \wedge(\mathsf{Q})) \otimes_K\operatorname{gr}\mathbb{K}}  ="1",
(40,0)  *++{\operatorname{gr}\mathbb{J}}    ="2",
(0,-18) *++{J \otimes_K \wedge(\mathsf{Q})} ="3",
(40,-18)*++{\operatorname{gr}(\mathbb{J} /\hspace{-1mm}/\mathbb{K}).} ="4",
{"1" \SelectTips{cm}{} \ar @{->}^-{\simeq} "2"},
{"1" \SelectTips{cm}{} \ar @{->>}_-{\otimes_{\operatorname{gr}\mathbb{K}} \operatorname{gr}\mathbb{K}/(\operatorname{gr}\mathbb{K})^{+}} "3"},
{"2" \SelectTips{cm}{} \ar @{->>} "4"},
{"3" \SelectTips{cm}{} \ar @{->}^-{\simeq} "4"},
\end{xy}
\]
Here one should notice 
\[ \operatorname{gr}((J \otimes \wedge(\mathsf{Q})) \otimes_{K} \mathbb{K}) = (J \otimes \wedge(\mathsf{Q})) \otimes_{K} \operatorname{gr}\mathbb{K}, \]
using an isomorphism $(\operatorname{gr}\mathbb{K} =)\, K \otimes \wedge(\mathsf{V}_{\mathbb{K}}) \overset{\simeq}{\longrightarrow} \mathbb{K}$ analogous to \eqref{E28}.
We see from the horizontal isomorphism at the top, which sends $(1_{J} \otimes 1_{\wedge(\mathsf{Q})}) \otimes_{K} 1_{\operatorname{gr}\mathbb{K}}$ to $1_{\operatorname{gr}\mathbb{J}}$, that $\operatorname{gr}\mathbb{K}$ is a graded Hopf sub-superalgebra of $\operatorname{gr}\mathbb{J}$; see also \cite[Remark 3.8]{M2}.
Moreover, the projection indicated by the vertical arrow on the RHS induces the isomorphism \eqref{E46}.
The horizontal isomorphism at the bottom proves the last statement.
\end{proof}

There is given in \cite[Proposition 3.14 (1)]{M2}
a description of $\mathbb{J}/\hspace{-1mm}/\mathbb{K}$ that looks similar to the one obtained 
in Theorem \ref{T41} (1) above.
But the cited result assumes that $\mathbb{J}$ is irreducible, and indeed, 
the down-to-earth proof given there is not 
valid in the present general case. 
The description of ours gives a simpler and more natural proof of the following result reproduced from \cite{M2}.

\begin{corollary}[\tu{\cite[Theorem 3.13 (2), (3)]{M2}}]\label{C41}
We have the following.
\begin{itemize}
\item[(1)]
$J/\hspace{-1mm}/K$ is naturally isomorphic to the largest purely even sub-super-coalgebra 
of $\mathbb{J}/\hspace{-1mm}/\mathbb{K}$.
\item[(2)]
$\mathsf{Q}$ is naturally isomorphic to the purely odd super-vector space
\[
P(\mathbb{J}/\hspace{-1mm}/\mathbb{K})_1 =\{\, z \in (\mathbb{J}/\hspace{-1mm}/\mathbb{K})_1 \mid 
\Delta_{\mathbb{J}/\hspace{-1mm}/\mathbb{K}}(z)=\overline{1}\otimes z + z \otimes \overline{1}\, \}
\]
of odd primitive elements of $\mathbb{J}/\hspace{-1mm}/\mathbb{K}$, where $\overline{1}=1_{\mathbb{J}}\operatorname{mod}(\mathbb{J}\mathbb{K}^+)$. 
\end{itemize}
\end{corollary}
\begin{proof}
(1) 
Notice from Proposition \ref{P42} and the remark preceding Example \ref{E31} that 
$J /\hspace{-1mm}/ K\, (=J\otimes_K \wedge^0(\mathsf{Q}))$ is the largest purely even sub-super-coalgebra of $J\otimes_K\wedge(\mathsf{Q})$. Then the desired result follows from 
the isomorphism proved by Theorem 4.5 (1) and its partial naturality on $J /\hspace{-1mm}/ K$.

(2)
Notice that taking $P(\ )$ is compatible with base extension. Then by using the isomorphism
\eqref{E42} after base extension to the algebraic closure of $\Bbbk$, we see
\[ P(J\otimes_K\wedge(\mathsf{Q}))_1= K\otimes_K \mathsf{Q}\, (=\mathsf{Q}). \]
The desired result follows from 
the isomorphism proved by Theorem 4.5 (1) and its partial naturality on 
$J \otimes_K \mathsf{Q}\, (=J\otimes_K \wedge^1(\mathsf{Q}))$. 
\end{proof}

As another advantage (from \cite{M2}) of our more conceptual treatment 
that uses the notion of co-free super-coalgebras, 
we have the following smoothness criteria for 
$\mathbb{J}/\hspace{-1mm}/\mathbb{K}$.

\begin{theorem}\label{T42}
We have the following.
\begin{itemize}
\item[(1)]
$\mathbb{J}/\hspace{-1mm}/\mathbb{K}$ is smooth if and only if so is $J/\hspace{-1mm}/K$. 
\item[(2)]
$\mathbb{J}/\hspace{-1mm}/\mathbb{K}$ is smooth if $\operatorname{char} \Bbbk=0$.
\end{itemize}
\end{theorem}
\begin{proof}
(1)\
By Theorem \ref{T41} we may suppose 
$\mathbb{J}/\hspace{-1mm}/\mathbb{K}
=
\operatorname{coF}^{J/\hspace{-1mm}/K}(J\otimes_K\mathsf{Q})$. 
We claim that the $J/\hspace{-1mm}/K$-comodule
$J\otimes_K\mathsf{Q}$ is injective, or equivalently,
co-flat; see \cite[Proposition A.2.1]{T0}. Indeed, the co-flatness (that is,
the right exactness of the associated co-tensor product
functor) follows, since the comodule turns into the co-free
comodule $J/\hspace{-1mm}/K\otimes \mathsf{Q}$ 
after base extension to the algebraic closure of $\Bbbk$;
see the proof of Proposition \ref{P42}. 
The claim, combined with
Proposition \ref{P32} and Remark \ref{R33},
proves the desired result. 

(2)\
By Part 1 it suffices to prove that
the coalgebra $J/\hspace{-1mm}/K$ is smooth in case
$\operatorname{char} \Bbbk=0$.
In view of \cite[Proposition 1.5]{FS} we may suppose by base extension such as above
that $\Bbbk$ is algebraically closed, in which case $J$ and $K$ are smash products 
(see Section \ref{Subsec2.5})
\[
J=\Bbbk G_J\ltimes U_J,\quad  K=\Bbbk G_K\ltimes U_K,
\]
where we have set
\[
G_L=G(L),\quad G_K=G(K),\quad U_J=U(P(L)),\quad U_K=U(P(K)). 
\]
 We see
\[
J/\hspace{-1mm}/K\simeq\Bbbk(G_L/G_K)\otimes 
U_L/\hspace{-1mm}/U_K
\]
as coalgebras. 
It remains to prove that $\Bbbk(G_L/G_K)$ and $U_L/\hspace{-1mm}/U_K$ are both smooth coalgebras. 
First, $\Bbbk(G_L/G_K)$, being spanned by
the grouplike elements of all right cosets of $G_L$ by $G_K$, is a smooth coalgebra. So is 
$U_L/\hspace{-1mm}/U_K$, since it is a
a pointed irreducible coalgebra of Birchkoff-Witt type, 
or more explicitly, it is isomorphic to the tensor product 
\[
U_L/\hspace{-1mm}/U_K\simeq \bigotimes_{\lambda \in \Lambda} C_{\lambda}
\]
of coalgebras $C_{\lambda}$, each of which is spanned by 
the divided power sequence of infinite length
\[
1, x_{\lambda}, \frac{1}{2}x_{\lambda}^2,\ \frac{1}{3!}x_{\lambda}^3,\ \cdots,\ \frac{1}{n!}x_{\lambda}^n,\ \cdots
\]
in $U_L$, where $(x_{\lambda})_{\lambda\in \Lambda}$ are arbitrarily chosen elements of $P(L)$ 
such that they modulo $P(K)$ form
a $\Bbbk$-basis of $P(L)/P(K)$. 
\end{proof}

\appendix
\section{On formal superschemes}
Let us translate some of the results obtained in the text into the language of formal superschemes. 
The authors do not know any literature which discusses in detail theory of formal schemes over
a field such as developed by \cite{T0}, in the generalized super context. 
But, at least basic definitions and results found in part of \cite{T0} are directly generalized as will be seen below. 
We continue to suppose $\operatorname{char} \Bbbk \ne 2$ and that super-coalgebras and Hopf superalgebras are super-cocommutative.

A set-valued (resp., group-valued) functor defined on 
the category $\mathsf{SAlg}_{\Bbbk}$ of super-commutative
superalgebras
is called a \emph{super $\Bbbk$-functor} (resp., \emph{super $\Bbbk$-group}); see \cite[Sect. 3]{MZ},
for example.

Let $\mathbb{C}$ be a super-coalgebra (over $\Bbbk$). 
Given an object $R\in \mathsf{SAlg}_{\Bbbk}$, $R\otimes \mathbb{C}$ is a super-coalgebra over $R$. Let
\[
G_R(R \otimes \mathbb{C})=\{\, g \in (R\otimes\mathbb{C})_0\mid \Delta(g)=g\otimes_R g,\ \varepsilon(g)=1_R\, \}
\]
denote the set of all even grouplike elements of $R\otimes \mathbb{C}$. This is naturally 
a group if $\mathbb{C}$ is a Hopf
superalgebra. 
One sees that $R \mapsto G_R(R \otimes \mathbb{C})$ gives rise to a functor, 
$\operatorname{Sp}^*(\mathbb{C})$,
which is called the \emph{formal superscheme} or \emph{formal supergroup} corresponding to $\mathbb{C}$. 
The formal superscheme $\operatorname{Sp}^*(\mathbb{C})$ is said to be \emph{smooth} if $\mathbb{C}$ is a
smooth super-coalgebra. 
Every formal supergroup is smooth if $\operatorname{char}\Bbbk=0$. 
See the paragraphs following Remark \ref{R32}.

The assignment $\mathbb{C}\mapsto \operatorname{Sp}^*(\mathbb{C})$ gives rise to a category equivalence
from the category of super-coalgebras (resp., Hopf superalgebras) to that full subcategory of the category of
super $\Bbbk$-functors (resp., super $\Bbbk$-groups) which consists
of all formal superschemes (resp., formal supergroups). The source category has (possibly, infinite) 
direct products given by the
tensor product $\otimes$, and the target full subcategory is closed under the direct product. 
By this fact on the category of super-coalgebras (resp., formal superschemes) one can define group objects
in the respective categories, which are precisely Hopf superalgebras (resp., formal supergroups); see \cite[Appendix B]{DNR}, for example. 

We remark that
the category of formal superschemes consists precisely 
of those super $\Bbbk$-functors $\mathfrak{X}$ such that
$\mathfrak{X}$ is isomorphic to the inductive limit associated to some filtered inductive system 
$(\mathfrak{X}_{\lambda}, \mathfrak{u}_{\lambda\mu})$ of finite affine superschemes, where each $\mathfrak{X}_{\lambda}$ is 
thus represented
by a finite-dimensional super-commutative superalgebra, say, $A_{\lambda}$, and the morphism 
$\mathfrak{u}_{\lambda\mu} : \mathfrak{X}_{\lambda}\to \mathfrak{X}_{\mu}$ for $\lambda <\mu$ is supposed to be a
closed immersion, or namely, to arise from a surjective superalgebra map $A_{\mu}\to A_{\lambda}$. 
It follows that the category is included in the category of sheaves in the fppf topology (faisceaux).
Also, a monomorphism 
$\operatorname{Sp}^*(\mathbb{C})\to \operatorname{Sp}^*(\mathbb{D})$ of formal superschemes uniquely
arises from an injective super-coalgebra map $\mathbb{C}\to \mathbb{D}$. Therefore, we may and we do 
call $\operatorname{Sp}^*(\mathbb{C})$ as a \emph{formal sub-superscheme} (resp., \emph{formal sub-supergroup}) of $\operatorname{Sp}^*(\mathbb{D})$ if $\mathbb{C}$ is a sub-super-coalgebra 
(resp., Hopf sub-superalgebra) of a super-coalgebra (resp., Hopf superalgebra) $\mathbb{D}$. 

Given a super $\Bbbk$-functor $\mathfrak{X}$,  we have a super $\Bbbk$-functor, $\mathfrak{X}_{\mathrm{ev}}$,
defined by
\[
\mathfrak{X}_{\mathrm{ev}}(R):=\mathfrak{X}(R_0),\quad R\in \mathsf{SAlg}_{\Bbbk}.
\] 
If $\mathfrak{X}$ is a formal superscheme or a formal supergroup, so is $\mathfrak{X}_{\mathrm{ev}}$. For if 
$\mathfrak{X}=\operatorname{Sp}^*(\mathbb{C})$, then $\mathfrak{X}_{\mathrm{ev}}=\operatorname{Sp}^*(C)$,
where $C$ is the largest purely even sub-super-coalgebra of $\mathbb{C}$; see \eqref{E25}. 
In this case, $\mathfrak{X}_{\mathrm{ev}}$ can be regarded as an ordinary formal scheme or group. 
In general, every
(ordinary) formal scheme, say, $\operatorname{Sp}^*(D)$ is regarded as a formal superscheme, with $D$
regarded as a (purely even) super-coalgebra. It is smooth at the same time as a formal scheme
and as a formal superscheme. 

Let
$
\mathfrak{G}=\operatorname{Sp}^*(\mathbb{J})
$
be a formal supergroup corresponding to a Hopf superalgebra $\mathbb{J}$. As in \eqref{E26a}, let $\mathsf{V}_{\mathbb{J}}=P(\mathbb{J})_1$
be the purely odd super-vector space of odd primitive elements of $\mathbb{J}$. One sees that 
$\mathfrak{G}_{\operatorname{ev}}$ is a formal sub-supergroup of $\mathfrak{G}$. 
The following results from the isomorphism $\phi_X$ in \eqref{E28}. 

\begin{prop}\label{PA1}
We have a $\mathfrak{G}_{\operatorname{ev}}$-equivariant isomorphism 
\[
\mathfrak{G}_{\operatorname{ev}}\times \operatorname{Sp}^*(\wedge(\mathsf{V}_{\mathbb{J}}))\simeq \mathfrak{G}
\]
of formal superschemes. 
\end{prop}

A group-equivariant object in the category of super-coalgebras is a coalgebra in the monoidal 
category ${}_{\mathbb{J}}\mathsf{SMod}$ (or $\mathsf{SMod}_{\mathbb{J}}$) of left (or right, according to the side on which $\mathbb{J}$ acts)
supermodules. On the other hand, such an object in the category of formal superschemes is called a 
\emph{left} or \emph{right formal $\mathfrak{G}$-superscheme}. Obviously, those objects in the respective categories 
are in a category equivalence. In Proposition \ref{PA1} above, $\mathfrak{G}$ is regarded as a left 
formal $\mathfrak{G}_{\operatorname{ev}}$-superscheme by multiplication. 

Retain
$\mathfrak{G}=\operatorname{Sp}^*(\mathbb{J})$ 
as above, and let $\mathfrak{H}=\operatorname{Sp}^*(\mathbb{K})$
be a formal sub-supergroup of $\mathfrak{G}$; 
thus, $\mathbb{K}$ is a Hopf sub-superalgebra of $\mathbb{J}$. 
Then for every $R \in \mathsf{SAlg}_{\Bbbk}$,
$\mathfrak{H}(R)$ is a subgroup of $\mathfrak{G}(R)$. It holds that $\mathfrak{H}$ is \emph{normal}, 
that is, for every $R$, $\mathfrak{H}(R)$ is a normal subgroup of $\mathfrak{G}(R)$, if and only if $\mathbb{K}$ is 
a normal Hopf sub-superalgebra of $\mathbb{J}$. 
Define
\[
\mathfrak{G}/\mathfrak{H}:=\operatorname{Sp}^*(\mathbb{J}/\hspace{-1mm}/\mathbb{K}), 
\]
where $\mathbb{J}/\hspace{-1mm}/\mathbb{K}$ is as in \eqref{E40a}; this is a formal supergroup
in case $\mathfrak{H}$ is normal. Then we have the co-equalizer diagram
\[
\mathfrak{G}\times \mathfrak{H}\rightrightarrows\mathfrak{G}\to \mathfrak{G}/\mathfrak{H}
\]
in the category of formal superschemes (formal supergroups, in case $\mathfrak{H}$ is normal), 
where the paired arrows indicate the multiplication and the projection. Main interest of ours is in the case where
$\mathfrak{H}$ is not normal in $\mathfrak{G}$. 

In case $\mathfrak{G}$ and hence $\mathfrak{H}$ are formal groups, the formal superscheme 
arising from the quotient formal scheme $\mathfrak{G}/\mathfrak{H}$ coincides with the 
quotient formal superscheme which one constructs, 
regarding $\mathfrak{G}$ and $\mathfrak{H}$ as formal supergroups. 

\begin{theorem}\label{TA1}
We have the following.
\begin{itemize}
\item[(1)] The morphism $\mathfrak{G}_{\operatorname{ev}}\to (\mathfrak{G}/\mathfrak{H})_{\operatorname{ev}}$
arising from the quotient $\mathfrak{G}\to \mathfrak{G}/\mathfrak{H}$ induces an isomorphism
\[
\mathfrak{G}_{\operatorname{ev}}/\mathfrak{H}_{\operatorname{ev}}
\overset{\simeq}{\longrightarrow}
 (\mathfrak{G}/\mathfrak{H})_{\operatorname{ev}}
\]
of formal schemes.
\item[(2)]
$\mathfrak{G}/\mathfrak{H}$ is smooth if and only if so is $\mathfrak{G}_{\operatorname{ev}}/\mathfrak{H}_{\operatorname{ev}}$.
\item[(3)]
$\mathfrak{G}/\mathfrak{H}$ is smooth if $\operatorname{char} \Bbbk=0$.  
\end{itemize}
\end{theorem}

Part 1 above is a restatement of Corollary \ref{C41} (1), while the rest is that of Theorem \ref{T42}. 

Recall $\mathsf{V}_{\mathbb{J}}=P(\mathbb{J})_1$. Let $\mathsf{V}_{\mathbb{K}}=P(\mathbb{K})_1$, and 
define $\mathsf{Q}:=\mathsf{V}_{\mathbb{J}}/\mathsf{V}_{\mathbb{K}}$, as in \eqref{E40b}. 
The isomorphisms obtained in Proposition \ref{P41} and Theorem \ref{T41} (1) are translated as follows. 

\begin{theorem}\label{TA2}
There is an $\mathfrak{H}$-equivariant isomorphism
\begin{equation}\label{EA1}
(\mathfrak{G}_{\operatorname{ev}}\times \operatorname{Sp}^*(\wedge(\mathsf{Q})))\times
^{\mathfrak{H}_{\operatorname{ev}}}\mathfrak{H}\simeq \mathfrak{G}
\end{equation}
of formal superschemes, which induces an isomorphism 
\begin{equation}\label{EA2}
\mathfrak{G}_{\operatorname{ev}}\times^{\mathfrak{H}_{\operatorname{ev}}} \operatorname{Sp}^*(\wedge(\mathsf{Q}))\simeq \mathfrak{G}/\mathfrak{H}
\end{equation}
of formal superschemes. 
\end{theorem}

In accordance with the fact (see Section \ref{Sec4}) that $\wedge(\mathsf{Q})$ is a coalgebra in $\mathsf{SMod}_K$
and in ${}_K\mathsf{SMod}$,  
we have regarded $\operatorname{Sp}^*(\wedge(\mathsf{Q}))$ as a right (for \eqref{EA1}) and left (for \eqref{EA2})
formal $\mathfrak{H}_{\operatorname{ev}}$-superschme,
where $\mathfrak{H}_{\operatorname{ev}}=\operatorname{Sp}^*(K)$.

The product $\times^{\mathfrak{H}_{\operatorname{ev}}}$ above is defined in general, as follows. 
Given a formal supergroup 
$\mathfrak{F}= \operatorname{Sp}^*(\mathbb{F})$, a right formal $\mathfrak{F}$-superscheme 
$\mathfrak{X}= \operatorname{Sp}^*(\mathbb{X})$ and a left formal $\mathfrak{F}$-superscheme 
$\mathfrak{Y}= \operatorname{Sp}^*(\mathbb{Y})$, we define $\mathfrak{X}\times^{\mathfrak{F}}\mathfrak{Y}$ by
\[
\mathfrak{X}\times^{\mathfrak{F}}\mathfrak{Y}:=\operatorname{Sp}^*(\mathbb{X}\otimes_{\mathbb{F}}\mathbb{Y}).
\] 
This superscheme is characterized by the co-equalizer diagram 
\[
\mathfrak{X}\times {\mathfrak{F}}\times \mathfrak{Y} \rightrightarrows \mathfrak{X}\times \mathfrak{Y}\to
\mathfrak{X}\times^{\mathfrak{F}}\mathfrak{Y}
\]
in the category of formal superschemes, where the paired arrows indicate the right $\mathfrak{F}$-action on 
$\mathfrak{X}$ and the left $\mathfrak{F}$-action on $\mathfrak{Y}$. 
In addition, for \eqref{EA1}, notice that 
$\mathfrak{X}\times^{\mathfrak{F}}\mathfrak{Y}$ is naturally a right
formal $\mathfrak{Y}$-superscheme,
in case $\mathfrak{Y}$ is a formal supergroup including $\mathfrak{F}$
as a formal sub-supergroup.

\begin{rem}
As analogues to Theorem \ref{TA2} above, 
Proposition 4.8 and Corollary 4.10 (see also Remark 4.13) of \cite{MT} prove results on affine algebraic supergroups, for which the situation is more complicated, so that the analogous isomorphisms hold only locally.
\end{rem}

\section*{Acknowledgments}
The authors thank the referee for helpful suggestions that improved the presentation of the paper.
The second-named author was supported by JSPS~KAKENHI, Grant Numbers 20K03552 and 23K03027.


\begin{thebibliography}{99}

\bibitem{CCF}
C.~Carmeli,\ L.~Caston,\ R.~Fioresi,\
\emph{Mathematical foundations of supersymmetry},\
EMS Series of Lectures in Mathematics,\ European Math. Soc.,\
Z\"{u}rich, 2011.

\bibitem{DNR}
S.~D\u{a}sc\u{a}lescu,\ C.~N\u{a}st\u{a}sescu,\ \c{S}.~Raianu,\ \emph{Hopf algebras: an introduction},\
Pure and Applied Mathematics vol.\, 235,\ Marcel Dekker, Inc.,\ New York/Basel, 2001. 

\bibitem{D}
Y.~Doi,\
\emph{Hopf extensions of algebras and Maschke type theorems},\ 
Israel\ J.\ Math.\ {\bf 72}(1--2) (1990), 99--108.

\bibitem{FS}
M.~A.~Farinati,\ A.~Solotar,\ 
\emph{Extensions of cocommutative coalgebras and a Hochschild-Kostant-Rosenberg
type theorem},\ 
Comm.\ Algebra {\bf 29}(10) (2001),\ 4721–4758.


\bibitem{HMT}
M.~Hoshi,\ A.~Masuoka,\ Y.~Takahashi,\
\emph{Hopf-algebra techniques applied to super Lie groups over a complete field},\
J. Algebra {\bf 562} (2020),\ 28--93.

\bibitem{Manin}
Y.~Manin,\
\emph{Gauge field theory and complex geometry},\
Translated from the 1984 Russian original by 
N. Koblitz and J. R. King,\ Second edition,\ 
Grundlehren der mathematischen Wissenschaften
vol.\, 289,\
Springer-Verlag, Berlin, 1997. 

\bibitem{M} 
A.~Masuoka,\
\emph{On Hopf algebras with cocommutative coradicals},\
J.\ Algebra \textbf{144} (1991), 451--466.

\bibitem{M1} 
A.~Masuoka,\
\emph{Hopf cohomology vanishing via approximation by Hochschild cohomology},\ 
Banach Center Publ. \textbf{61} (2003), 111--123.

\bibitem{M2} 
A.~Masuoka,\
\emph{The fundamental correspondences in super affine groups and
super formal groups},\
J.\ Pure\ Appl.\ Algebra \textbf{202}(1--3) (2005), 284--312.

\bibitem{M3} 
A.~Masuoka,\ 
{\em Harish-Chandra pairs for algebraic affine supergroup schemes over an arbitrary field},\ 
Transform. Groups \textbf{17}(4) (2012), 1085--1121.
 
\bibitem{MO} 
A.~Masuoka,\ T.~Oka,\ 
\emph{Unipotent algebraic affine supergroups and nilpotent Lie superalgebras},\
Algebr. Represent. Theory \textbf{8}(3) (2005), 397--413. 

\bibitem{MT}
A.~Masuoka,\ Y.~Takahashi,\ \emph{Geometric construction of quotients $G/H$ in
supersymmetry},\ Transform. Groups {\bf 26}(1) (2021), 347--375.

\bibitem{MZ} A.~Masuoka,\ A.~N.~Zubkov,\ 
\emph{Quotient sheaves of algebraic supergroups are superschemes},\ 
J.\ Algebra \textbf{348} (2011), 135--170.

\bibitem{MZ1} A.~Masuoka,\ A.~N.~Zubkov,\ 
\emph{Group superschemes},\  
J. Algebra \textbf{605} (2022),  89--145.

\bibitem{Sw} 
M.~E.~Sweedler,\ 
\emph{Hopf Algebras},\
W. A. Benjamin, Inc., New York,\ 1969.

\bibitem{T0} M.~Takeuchi,\ 
\emph{Formal schemes over fields},\ 
Comm.\ Algebra \textbf{5}(14) (1979),\ 1483--1528.


\bibitem{T} M.~Takeuchi,\ 
\emph{Relative Hopf modules—equivalences and freeness criteria},\ 
J.\ Algebra \textbf{60} (1979),\ 452–471.






\end{thebibliography}
\end{document}